 \documentclass[11pt]{article}
\usepackage{amssymb, amsthm, amsmath, amscd}
\newtheorem{thm}{Theorem}[section]
\newtheorem{lem}[thm]{Lemma}
\newtheorem{prop}[thm]{Proposition}
\newtheorem{cor}[thm]{Corollary}
\newtheorem{NN}[thm]{}
\theoremstyle{definition}\newtheorem{df}[thm]{Definition}
\theoremstyle{definition}\newtheorem{rem}[thm]{Remark}
\theoremstyle{definition}

\renewcommand{\phi}{\varphi}

\newcommand{\Z}{\mathbb{Z}}
\newcommand{\Q}{\mathbb{Q}}

\newcommand{\C}{\mathbb{C}}
\newcommand{\T}{\mathbb{T}}

\newcommand{\Aff}{\operatorname{Aff}}

\newcommand{\id}{\operatorname{id}}

\newcommand{\morp}{contractive completely positive linear map}

\newcommand{\hm}{homomorphism}
\newcommand{\dt}{\delta}
\newcommand{\ep}{\epsilon}
\newcommand{\F}{{\cal F}}

\usepackage{color}

\newcommand{\andeqn}{\,\,\,{\rm and}\,\,\,}
\newcommand{\rforal}{\,\,\,{\rm for\,\,\,all}\,\,\,}
\newcommand{\CA}{C*-algebra}
\newcommand{\SCA}{C*-subalgebra}
\newcommand{\tr}{{\rm TR}}

\newcommand{\bt}{{\beta}}

\newcommand{\beq}{\begin{eqnarray}}
\newcommand{\eneq}{\end{eqnarray}}
\newcommand{\tforal}{\,\,\,\text{for\,\,\,all}\,\,\,}
\newcommand{\tand}{\,\,\,\text{and}\,\,\,}

\title{Asymptotic unitary equivalence in $C^*$-algebras }
\author{Huaxin Lin and Zhuang Niu
 }
\date{}

\begin{document}

\maketitle

\begin{abstract}
Let $C=C(X)$ be the unital \CA\, of all continuous functions on a finite CW  complex $X$ and  let $A$ be a unital simple
\CA\, with tracial rank at most one.
We show that two unital monomorphisms $\phi, \psi: C\to A$ are asymptotically unitarily equivalent,\, i.e., there exists a continuous path of
 unitaries $\{u_t: t\in [0,1)\}\subset A$ such that
 $$
 \lim_{t\to 1} u_t^*\phi(f)u_t=\psi(f)\tforal f\in C(X),
 $$
 if and only if
\beq\nonumber
[\phi]&=&[\psi]\,\,\,{\rm in}\,\,\, KK(C, A),\\\nonumber
\tau\circ \phi&=&\tau\circ \psi\tforal \tau\in T(A),\andeqn\\\nonumber
\phi^{\dag}&=&\psi^{\dag},
\eneq
where $T(A)$ is the simplex of tracial states of $A$ and
$\phi^{\dag}, \psi^{\dag}: U(M_{\infty}(C))/DU(M_{\infty}(C))\to$ \\ $U(M_{\infty}(A))/DU(M_{\infty}(A))$ are induced \hm s and
where $U(M_{\infty}(A))$ and $U(M_{\infty}(C))$ are groups of union of unitary groups of $M_k(A)$ and $M_k(C)$ for all integer $k\ge 1,$ $DU(M_{\infty}(A))$ and $DU(M_{\infty}(C))$ are commutator subgroups of $U(M_{\infty}(A))$ and $U(M_{\infty}(C)),$ respectively.
We actually prove a more general result for the case that $C$ is any general unital AH-algebra.

\end{abstract}

\section{Introduction}
In the study of topology, it is fundamentally important to study
continuous maps between topological spaces. In the study of \CA s,
or sometime called the non-commutative topological space, it is essential to study
\hm s from one \CA\,to another.

One of the central problems in classification of amenable C*-algebras is to determine how certain equivalence classes of homomorphisms between C*-algebras can be determined by their K-theoretical invariants.
In this note, we will study the unital monomorphisms from a {unital} commutative C*-algebra $C,$ or, more general, arbitrary {unital} AH-algebras, to a simple C*-algebra $A$
with finite tracial rank (see \ref{TA1} below)
and consider the question when two given
unital monomorphisms $\phi, \psi: C\to A$ are asymptotically unitarily equivalent, that is,  when does there exist a continuous path of unitaries
$\{u_t: t\in [0,1)\}\subset A$ such that
$$
\lim_{t\to 1} u_t^*\phi(f)u_t=\psi(f)\tforal f\in C.
$$

If one considers approximately unitary equivalence (recall that the maps $\phi$ and $\psi$ are approximately unitarily equivalent if there exists a sequence of unitaies $\{u_n\}\subset A$ such that $\lim_{n\to\infty} u_n^*\phi(f)u_n=\psi(f)\tforal f\in C$), there are  already several results recently:

Let $C$ be a unital AH-algebra and let $A$ be a unital simple \CA\, with tracial rank zero. It has been shown in \cite{Lin-UTAF} by the first author   that $\phi$ and $\psi$ are approximately unitarily equivalent if and only if $$
[\phi]=[\psi]\,\,\,{\rm in}\,\,\rm KL(C,A)\andeqn
\tau\circ \phi=\tau\circ \psi\tforal \tau\in T(A).
$$
And in \cite{Ng-W-IUMJ}, Ng and Winter showed that the result above still holds if $C=C(X)$ with $X$ a second countable, path connected, compact metric space and $A$ is any simple unital separable nuclear C*-algebra which is real rank zero and $\mathcal Z$-stable, where $\mathcal Z$ is the Jiang-Su algebra.

Beyond the real rank zero case, in a more recent paper (\cite{Lin-AU11}), it was shown that, if $A$ is a unital simple \CA\, with tracial rank at most one, then $\phi$ and $\psi$ are approximately unitarily equivalent
if and only if
\beq\nonumber
[\phi]&=&[\psi]\,\,\,{\rm in}\,\,\, KL(C,A),\\
\tau\circ \phi&=&\tau\circ \psi\tforal \tau\in T(A)\andeqn\,
\phi^{\ddag}=\psi^{\ddag},
\eneq
where $\phi^{\ddag},\,\psi^{\ddag}: U(M_{\infty}(C)/\overline{DU(M_{\infty}(C))}\to
U(M_{\infty}(A))/{\overline{DU(M_{\infty}(A))}}$ are induced \hm s and
$DU(M_{\infty}(C))$ and $DU(M_{\infty}(A))$ are commutator subgroups of $
\cup_{k=1}^{\infty}U(M_k(C))$ and $\cup_{k=1}^{\infty}U(M_k(A)),$ respectively.

These results play important roles in the recent progress of the Elliott program of the classification of amenable \CA s. It is natural to ask whether approximate
unitary equivalence is the same as asymptotic unitary equivalence.
It turns out, from a result of Kishimoto and Kumjian (\cite{KK}), that, in general, asymptotic unitary equivalence is different from approximate unitary equivalence. In particular, they studied the case that both $A$ and $C$ are unital simple A$\T$-algebras of real rank zero.

Then, in \cite{Lnhomtp}, the following criterion for asymptotical unitarily equivalent was developed for any unital AH-algebra $C$ and any simple C*-algebra $A$ with tracial rank zero:  Suppose that $\phi, \psi: C\to A$ are two unital monomorphisms. Then
$\phi$ and $\psi$ are asymptotically unitarily equivalent if and only
if
\beq\nonumber
[\phi]=[\psi]\,\,\, {\rm in}\,\,\, KK(C,A),\\
\tau\circ \phi=\tau\circ \psi\tforal \tau\in T(A),\\
\overline{R_{\phi, \psi}}=0,
\eneq
where $R_{\phi, \psi}$ is the rotation map, which will be defined in \ref{Rot-M}.

It { is} worth to point out that one application of this result is to the study of Voiculescu's AF-embedding
problem: Let $\Omega$ be a compact metric space and let $G$ be a finitely generated abelian group. Suppose that $\Lambda$ is a $G$ action on $X.$ Then the above mentioned result can be used to prove that $C(\Omega)\rtimes_\Lambda G$ can be embedded into a unital simple AF-algebra if and only if $\Omega$ has a faithful $\Lambda$-invariant Borel probability measure.

There are other applications.  With a method developed by Winter (\cite{Winter-Z}),
the above mentioned asymptotic unitary equivalence result was also used to give
an important advance in the Elliott program (see \cite{Winter-Z}, \cite{Lin-App} and \cite{L-N}) for the C*-algebras which might be projectionless.  An even further advance was made which allows the class of unital separable amenable simple \CA s classified by the conventional
Elliott invariant to include \CA s which are so-called rationally finite tracial rank and their $K_0$-groups may not have the Riesz interpolation property.  The technical key of this advance was the following asymptotic unitary equivalence theorem.
\begin{thm}[Theorem 7.2, \cite{Lnclasn}]\nonumber
Let $C$ be a unital simple AH-algebra of slow dimension growth and let $A$ be any unital simple \CA\, with tracial rank at most one. Suppose that $\phi, \psi: C\to A$ are two unital monomorphisms. Then
$\phi$ and $\psi$ are asymptotically unitarily equivalent if and only if
\beq\nonumber
[\phi]=[\psi]\,\,\, {\rm in}\,\,\, KK(C,A),\\\nonumber
\tau\circ \phi=\tau\circ \psi\tforal \tau\in T(A),\\\nonumber
\phi^{\ddag}=\psi^{\ddag},\andeqn\,
\overline{R_{\phi, \psi}}=0.
\eneq
\end{thm}

However, while in \cite{Lnclasn}, the theorem above also was proved for certain non-simple AH-algebras, it  only includes those unital AH-algebras whose K-theory behave as low dimensional topological spaces.
In this paper we will generalize the theorem above so that it will apply to all unital AH-algebras (with no restriction on dimension growth). {In particular, it holds for $C=C(X)$ for any compact metric space $X.$}

Moreover, in the case that $K_1(C)$ is finitely generated, we also find that the invariant
could be simplified. In fact, in Theorem \ref{MT2} below, the conditions that $\overline{R_{\phi, \psi}}=0$ and $\phi^{\ddag}=\psi^{\ddag}$  can be simplified to the condition
that $\phi^{\dag}=\psi^{\dag},$ i.e.,
$\phi$ and $\psi$ induce the same \hm s on
$\cup_{k=1}^{\infty}U(M_k(C))/DU(M_{\infty}(C))$.
However, we also point out that, in general, this simplification is not possible.  A specific example will be presented.

{\bf Acknowledgement}
Most part of this work was done when both authors
were in East China Normal University during 2012 Operator Algebra Program. The research of Z.~N.~is supported by an NSERC Discovery grant.

\section{Preliminaries}

\begin{NN}\label{NN1}
{\rm Let $A$ be a unital stably finite \CA. Denote by $\mathrm{T}(A)$ the simplex of tracial
states of $A$ and denote by $\textrm{Aff(T}(A))$ the space of all real
affine continuous functions on $\mathrm{T}(A).$ Suppose that $\tau\in \mathrm{T}(A)$ is a
tracial state. We will also {denote by} $\tau$ the trace $\tau\otimes
\mathrm{Tr}$ on $\textrm{M}_k(A)=A\otimes \textrm{M}_k(\C)$ (for every integer $k\ge 1$), where
$\mathrm{Tr}$ is the standard trace on $\mathrm{M}_k(\C).$


Denote by $\textrm{M}_{\infty}(A)$ the set $\displaystyle{\bigcup_{k=1}^{\infty}\textrm{M}_k(A)},$ where $\textrm{M}_k(A)$ is regarded as a C*-subalgebra of $\textrm{M}_{k+1}(A)$ by the embedding $\textrm{M}_k(A)\ni a\mapsto \left(\begin{array}{cc}a&0\\0&0\end{array}\right)\in \textrm{M}_{k+1}(A).$

{For any projection $p\in \mathrm{M}_\infty(A)$,
the evaluation $\tau\mapsto\tau(p)$ defines a positive affine function on
$\mathrm{T}(A)$. This induces a canonical positive homomorphism
$\rho_A: K_0(A)\to \textrm{Aff(T}(A))$.}

Denote by $\mathrm{S}(A):=C_0((0, 1))\otimes A$ the suspension of $A$, denote by $\mathrm{U}(A)$ the unitary group of $A$, and denote by $\mathrm{U}(A)_0$ the connected component of $\mathrm{U}(A)$ containing the identity.

{Let} $C$ be another unital \CA\, and {let}  $\phi: C\to A$ be a unital *-homomorphism. Denote by $\phi_{\mathrm{T}}: \mathrm{T}(A)\to \mathrm{T}(C)$
the continuous affine map induced by $\phi,$ i.e., $$\phi_\mathrm{T}(\tau)(c)=\tau\circ \phi(c)$$ for all $c\in C$ and $\tau\in \mathrm{T}(A).$ Denote by $\phi_\sharp:\mathrm{Aff}(\mathrm{T}(C))\to \mathrm{Aff}(\mathrm{T}(A))$ {the map defined} by $$\phi_\sharp(f)(\tau)=f(\phi_{\mathrm{T}}(\tau))$$ for all $\tau\in\mathrm{T}(A)$.
}
\end{NN}

\begin{df}\label{ddag}
{\rm Let $A$ be a unital \CA.
Denote by $DU(A)$ the subgroup of generated by the commutators of $U(A)$ and
denote by $CU(A)$ the closure of $DU(A).$   If $u\in U(A),$ {its} image in {the quotient} $U(A)/CU(A)$ will be denoted by $\overline{u}.$

Let $B$ be another unital \CA\, and let $\phi: A\to B$ be a unital \hm.
{It} is clear that $\phi$ maps $CU(A)$ into $CU(B).$ Let $\phi^{\ddag}$ denote the} induced \hm\ from $U(A)/CU(A)$
into $U(B)/CU(B)$.  It is also clear that $\phi$ maps $DU(A)$ into $DU(B).$ Denote by $\phi^{\dag}: U(A)/DU(A)\to U(B)/DU(B)$ the \hm\, induced by
$\phi.$

Let $n\ge 1$ be any integer. Denote by $U_n(A)$ the unitary group
of $M_n(A),$ and denote by $DU_n(A)$ and $CU_n(A)$ the commutator
subgroup of $U_n(A)$ and its closure, respectively.
 {Regard} $U_n(A)$ as a subgroup of $U_{n+1}(A)$ via the embedding 
{$\textrm{U}_k(A)\ni u\mapsto \left(\begin{array}{cc}u&0\\0&1\end{array}\right)\in \textrm{U}_{k+1}(A),$}and denote by $U_{\infty}(A)$ the union
of {all} $U_n(A).$

Consider the union $CU_\infty(A):=\bigcup_n CU_n(A)$. It is then a normal subgroup of $U_\infty(A)$, and the quotient $U(A)_{\infty}/CU_{\infty}(A)$ is in fact isomorphic to the inductive limit
of $U_n(A)/CU_n(A)$ (as abelian groups).
Similarly, $DU_{\infty}(A):=\bigcup_n DU(A)_n$ is a normal subgroup  of
$U_{\infty}(A).$
We will use $\phi^{\ddag}$ for the \hm\ {induced by $\phi$} from
$U_{\infty}(A)/CU_{\infty}(A)$ into $U_{\infty}(B)/CU_{\infty}(B)$, and
we will use $\phi^{\dag}$ for the \hm\, induced by $\phi$ from
$U_{\infty}(A)/DU_\infty (A).$

\end{df}

\begin{rem}\label{alg-iso}
By Corollary 3.5 of \cite{Lin-hmtp}, {if $A$ has tracial rank at most one (see \ref{TA1} below)}, the map natural map $$U(A)/CU(A)\to U(M_n(A))/CU(M_n(A))$$ is an isomorphism for any integer $n\geq 1$.
\end{rem}

\begin{df}\label{DDet}
Let $A$ be a unital \CA, and {let} $u\in U(A)_0.$ {Let} $u(t)\in C([0,1],A)$ be
a piecewise-smooth path of unitaries such that $u(0)=u$ and $u(1)=1.$
Then the de la Harpe--Skandalis determinant of $u(t)$ is defined by
$$
{\mathrm{Det}}(u(t))(\tau)=\frac{1}{2\pi i}\int_0^1\tau({{d}u(t)\over{{d}t}}u(t)^*) {d}t \tforal \tau\in T(A),
$$
{{which induces}}
 a \hm\, $${\mathrm{\overline{Det}}:U(A)_0\to \Aff(T(A))/\overline{\rho_A(K_0(A))}}.$$

{The determinant ${\overline{{\mathrm{Det}}}}$ can be extended to a map from} $U_{\infty}(A)_0$ into $\Aff(T(A))/\overline{\rho_A(K_0(A))}.$
{It is easy to see that the determinant} vanishes on the {closure} of commutator subgroup
of $U_{\infty}(A).$ {In fact, by 3.1 of \cite{Thomsen-rims},
the closure of the commutator subgroup is exactly the kernel of this map,
that is, it induces an isomorphism} $\overline{{\mathrm{Det}}}:
U_{\infty}(A)_0/CU_{\infty}(A)\to \Aff(T(A))/\overline{\rho_A(K_0(A))}.$
{Moreover, by (\cite{Thomsen-rims}),
one has the following short exact sequence}
\beq\label{Texact}
 0\to \Aff(T(A))/\overline{\rho_A(K_0(A))}{\to} U_{\infty}(A)/CU_{\infty}(A){\stackrel{\varPi}{\to}} K_1(A)\to 0
\eneq
which splits ({where the embedding of $\Aff(T(A))/\overline{\rho_A(K_0(A))}$ induced by $(\overline{{\mathrm{Det}}})^{-1}$}). We will fix a splitting map $s_1: K_1(A)\to U_{\infty}(A)/CU_{\infty}(A).$  {The notation $\varPi$ and} $s_1$
will be used late without further warning.
For each $\bar{u}\in s_1(K_1(A))$, select and fix one element $u_c\in\bigcup_{n=1}^\infty M_n(A)$ such that $\overline{u_c}=\bar{u}$. Denote this set by $U_c(A)$.
Moreover, in the case that {$A$ is unital, simple and} $TR(A)\leq 1$ (see \ref{TA1} below), one has that $U(A)/U_0(A)$ to $K_1(A)$ is an isomorphism and $\overline{\mathrm{Det}}: U_0(A)/CU(A)\to \Aff(T(A))/\overline{\rho_A(K_0(A))}$ is also an isomorphism. Then one has  
\beq\label{Texact-2}
 0\to \Aff(T(A))/\overline{\rho_A(K_0(A))}\to U(A)/CU(A)\to K_1(A)\to 0.
\eneq
\end{df}



%

\begin{df}\label{DKL}

Let $A$ be a  unital \CA\, and let $C$ be a  separable \CA\, which satisfies the Universal Coefficient Theorem.
{Recall that $KL(C, A)$ is the quotient of $KK(C, A)$ modulo pure extensions.}  By {a} result of D{\u a}d{\u a}rlat and Loring in \cite{DL}, {one has}
\beq\label{N2-1}
{KL}(C,A)=\mathrm{Hom}_{\Lambda}(\underline{{K}}(C), \underline{{K}}(A)),
\eneq
where
$$
\underline{{K}}(B)=({K}_0(B)\oplus{K}_1(B))\oplus( \bigoplus_{n=2}^{\infty}({K}_0(B,\Z/n\Z)\oplus{K}_1(B,\Z/n\Z) ) )
$$
for any \CA\, $B.$ Then, in the rest of the paper,
we will identify $KL(C, A)$ with
$\mathrm{Hom}_{\Lambda}(\underline{{K}}(C), \underline{{K}}(A))$.
Let $\kappa\in KL(C,A).$
Denote by $\kappa_i: K_i(C)\to K_i(A)$ the \hm\,
given by $\kappa$ with ${i=0,1.}$
\end{df}




\begin{df}\label{TA1}
Let $k\ge 0$ be an integer.
A unital simple C*-algebra $A$ has tracial rank at most $k,$
denoted by $\mathrm{TR}(A)\leq k$, if for any finite subset
$\mathcal F\subset A$, any $\ep>0$, and nonzero $a\in A^+$,
there exist {a} nonzero projection $p\in A$ and a C*-subalgebra
{$I\cong \bigoplus_{i=1}^m C(X_i)\otimes M_{r(i)}$}
with $1_I=p$ for some finite CW-complexes $X_i$ with dimension at most  $k$
such that
\begin{enumerate}
\item $\|xp-xp\|<\ep$ for any $x\in\mathcal F$,
\item for any $x\in\mathcal F$, there is $x'\in I$ such that $\|pxp-x'\|
    \leq\ep$, and
\item $1-p$ is Murray-von Neumann equivalent to a projection in $\overline{aAa}$.
    \end{enumerate}

{Moreover}, if {the C*-subalgebra} $I$ {above} can be chosen to be a finite dimensional \CA,
then $A$ is said to have tracial rank zero, and in such case, we write
$\mathrm{TR}(A)=0.$ It is a theorem of Guihua Gong \cite{Gong-AH} that every unital simple AH-algebra
with no dimension growth has tracial rank  at most one.
It has been proved in \cite{Lnclasn} that every ${\cal Z}$-stable unital simple AH-algebra
has tracial rank at most one.
It is shown recently (\cite{Lin-LAH}) that if a unital separable
simple \CA\, $A$ satisfying the UCT has $TR(A)\le k,$ then $TR(A)\le 1$.
\end{df}

\begin{df}\label{D-Maf}
{\rm

Let $A$ and $B$ be two unital C*-algebras, and let $\psi$ and $\phi$ be two
unital monomorphisms from $B$ to $A$. Then the mapping torus $M_{\phi, \psi}$
is the C*-algebra defined by $$M_{\phi, \psi}:=\{f\in\textrm{C}([0, 1], A);\ f(0)=\phi(b)\ \textrm{and}\ f(1)=\psi(b)\ \textrm{for some}\ b\in B\}.$$
%
%
For any $\psi, \phi\in\mathrm{Hom}(B, A)$, denoting by $\pi_0$ the evaluation of $M_{\phi, \psi}$ at $0$, we have the short exact sequence
\begin{displaymath}
0\to \mathrm{S}(A) \to M_{\phi,\psi}
\to^{\pi_0} B \to 0.
\end{displaymath}
If {$\phi_{*i}=\psi_{*i}$ ($i=0,1$),}
then the corresponding six-term exact sequence breaks down to the following two extensions:
\begin{displaymath}
\eta_{{i}}(M_{\phi, \psi}):\
0\to K_{i+1}(A)\to  K_{{i}}(M_{\phi, \psi})\to  K_{i}(B)\to 0\\
\,\,\,{(i=0,1).}
\end{displaymath}
}
\end{df}

\begin{NN}\label{Rot-M}
{\rm
Suppose that, in addition,
\beq\label{Dr-2}
\tau\circ \phi=\tau\circ \psi\tforal \tau\in \textrm{T}(A).
\eneq
For any {continuous}  piecewise smooth path of unitaries $u(t)\in M_{\phi, \psi}$, consider the path of unitaries $w(t)=u^*(0)u(t)$ in $A$. Then it is a continuous {and piecewise smooth}  path with $w(0)=1$ {and
$w(1)=u^*(0)u(1).$}  Denote by $R_{\phi, \psi}(u)=\mathrm{Det}(w)$ the determinant of $w(t)$. 
It is clear {{with the assumption of (\ref{Dr-2})}} that $R_{\phi, \psi}(u)$ depends only on the homotopy class of $u(t)$. Therefore, it induces a {\hm}, denoted by $R_{\phi, \psi}$, from $K_1(M_{\phi, \psi})$ to $\mathrm{Aff}(\mathrm{T}(A)).$
One has the following lemma.
}
\end{NN}


%

\begin{lem}[3.3 of \cite{Lin-Asy}, also see \cite{KK}]\label{DrL} When
(\ref{Dr-2}) holds, the following diagram commutes:
$$
\begin{array}{ccccc}
K_0(A) && \stackrel{[\imath]_1}{\longrightarrow} && K_1(M_{\phi, \psi})\\
& \rho_A\searrow && \swarrow R_{\phi,\psi} \\
& & \mathrm{Aff(T}(A)) \\
\end{array}
$$
\end{lem}

\begin{df}\label{R0}
Fix two unital C*-algebras $A$ and $B$ with $\tr(A)\neq\O$. {Define} $\mathcal R_0$ to be the subset of $\mathrm{Hom}(K_1(B), \Aff(T(A)))$ consisting of those homomorphisms $h\in \mathrm{Hom}(K_1(B), \Aff(T(A)))$
{for which}  there exists a homomorphism $d: K_1(B)\to K_0(A)$ such that $$h=\rho_A\circ d.$$ It is clear that $\mathcal R_0$ is a subgroup of $\mathrm{Hom}(K_1(B), \Aff(T(A)))$.
\end{df}

\begin{NN}\label{R-bar}
{\rm
If $[\phi]=[\psi]$ in $KK(B, A)$, {then} the exact sequences $\eta_i(M_{\phi, \psi})$ ($i=0, 1$) split. In particular, there is a lifting $\theta: K_1(B)\to K_1(M_{{\phi, \psi}})$. Consider the map $$R_{\phi, \psi}\circ\theta: K_1(B)\to\Aff(T(A)).$$
If a different lifting $\theta'$ is chosen, then,
$\theta-\theta'$ maps  $K_1(B)$ into  $K_0(A).$ Therefore  $$R_{\phi, \psi}\circ\theta-R_{\phi, \psi}\circ\theta'\in \mathcal{R}_0.$$
Then {define} $$\overline{R}_{\phi, \psi}=[R_{\phi, \psi}\circ\theta]\in \mathrm{Hom}(K_1(B), \Aff(T(A)))/\mathcal R_0.$$
%
%
See 3.4 of \cite{Lnclasn} for more details.
}
\end{NN}

\section{A basic homotopy lemma}
The following is taken from Lemma 2.8 of \cite{Lin-hmtp}.
\begin{lem}\label{pert-comm-lem}
Let $C$ be a unital nuclear C*-algebra. Let $\mathcal F\subseteq C$ be a finite subset, $N\in\mathbb{N}$, and $\ep>0$. There then exist a finite subset $\mathcal G\subseteq C$ and $\delta>0$ such that for any unital C*-algebra $A$, any unitary $u\in A$ and any unital homomorphism $\phi: C\to A$ with $$\|[\phi(c), u]\|<\delta,\quad\forall c\in\mathcal G,$$ there is a unital completely positive linear map $L: C\otimes\mathrm{C}(\mathbb T)\to A$ such that
$$\|L(f\otimes z^n)-\phi(f)u^n\|<\ep
, \quad\forall f\in\mathcal F,\ -N\leq n\leq N.$$
\end{lem}

Let $X$ be a metric space. In the rest of the paper, we fix the metric on $X\times \T$
to be
$$
{\mathrm{dist}}((x,t), (y,s))=\sqrt{{\mathrm{dist}}(x,y)^2+{\mathrm{dist}}(t,s)^2},\quad\forall x, y\in X,\ s, t\in\mathbb T.
$$

\begin{df}[5.2 of \cite{LinTAI}]
Recall that
a unital simple C*-algebra $A$ is said to be tracially approximately divisible if for any finite subset $\F\subseteq A$, any $\ep>0$, any natural number $N$, and any $a\in A^+$, there is a C*-subalgebra $B\subseteq A$ with $B\cong M_k(\C)$ for some $k\geq N$ such that if $p=1_B$, then
\begin{enumerate}
\item[(1)] $\F\subseteq_\ep B'\cap A$, and
\item[(2)] $1-p$ is Murray-von Neumann equivalent to a projection in $\overline{aAa}$,
\end{enumerate}
where $B'\cap A$ is the relative commutant of $B$ in $A$.
\end{df}

\begin{rem}
The definition above is slightly different---but equivalent---to the original definition in \cite{LinTAI}, in which the first condition is replaced by
\begin{enumerate}
\item[(1')] $\|cf-fc\|<\ep$ for any $f\in\F$ and any $c$ in the unit ball of $B$.
\end{enumerate}
Indeed, as in \cite{BKR-ADiv}, for any finite dimensional C*-algebra $B\subseteq A$, one considers  the conditional expectation
$$\mathbb{E}_B: A\ni a\mapsto \int_{U(B)}u^*au d\mu,$$ where $\mu$ is the Haar measure on the unitary group $U(B)$. It is clear that $\mathbb{E}_B(a)$ commutes with $B$.
Now, if $f\in A$ satisfies $\|fc-cf\|<\ep$ for any $c$ in the unit ball of $B$, one has that $$\|\mathbb{E}_B(f)-f\|<\ep.$$ In particular, this implies that $f\in_\ep B'\cap A$, and shows that the two definitions of tracially approximate divisibility are equivalent.
\end{rem}

Similar to \cite{Lin-hmtp}, for any nondecreasing function $\Delta: (0, 1)\to(0, 1)$ with $\lim_{t\to 0}\Delta(t)=0$, define
$$\Delta_{00}(t)=
\Delta(\frac{1}{2^{n+1}}),\quad \textrm{if $t\in [\frac{1}{2^{n+1}}, \frac{1}{2^{n}})$},
$$
and
$$ \Delta_{0}(t)=\frac{\sqrt{2}}{48}\Delta_0(t\sqrt{2}/6)t$$
Then $\Delta_{00}$ and $\Delta_0$ are also nondecreasing and satisfy $\lim_{t\to 0}\Delta_{00}(t)=0$, $\lim_{t\to 0}\Delta_0(t)=0$.

\begin{df}\label{mea}
Let $X$ be a compact metric space and $P\in M_r(C(X))$ be a projection, where $r\geq 1$ is an integer. Put $C=PM_r(C(X))P$. Suppose $\tau\in T(C)$. It is known that there exists a probability measure $\mu_\tau$ on $X$ such that
$$\tau(f)=\int_X t_x(f(x)) d\mu_\tau(x),$$
where $t_x$ is the normalized trace on $P(x)M_rP(x)$ for all $x\in X$.
\end{df}

\begin{rem}
Regard $C(X)$ as the center of $C=PM_r(C(X))P$, and denote by $\iota: C(X)\to C$ the embedding. Then the measure $\mu_\tau$ is in fact induced by the trace $\tau\circ\iota$ on $C(X)$.
\end{rem}

\begin{rem}
The C*-algebra $(PM_r(C(X))P)\otimes C(\mathbb T)$ is isomorphic to the homogeneous C*-algebra $\tilde{P}M_r(C(X\times\mathbb T))\tilde{P}$ with the projection $\tilde{P}$ given by $\tilde{P}(x, z)=P(x)$. Hence there is a natural embedding of $C(X\times\mathbb T)$ into $(PM_r(C(X))P)\otimes C(\mathbb T)$ as the center.
\end{rem}

\begin{lem}\label{pert-meas-lem}
Let $C=PM_r(C(X))P$ for some compact metrizable space $X$, and let $\Delta: (0, 1)\to(0, 1)$ be a non-decreasing function and
$1>\eta>0.$ 
Let $\mathcal F\subseteq C$, $\mathcal G'\subseteq C\otimes \mathrm{C}(\mathbb T)$, $\mathcal H\subseteq C\otimes \mathrm{C}(\mathbb T)$ be finite subsets, and let $\ep>0$. Then there are $\delta>0$ and a finite subset $\mathcal G\subseteq C$ such that for any C*-algebra $A$ which is tracially approximately divisible, any homomorphism $\phi: C\to A$, any unitary $u\in A$ with
\beq\nonumber
\|[\phi(c), u]\|<\delta\quad\forall c\in\mathcal G \tand\\
\mu_{\tau\circ\phi}(O_a)> \Delta(a) \tforal \tau\in T(A)
  \eneq
  and for any open ball $O_a$
of $X$ with radius $a>\eta$
there exist unitaries $w_1, w_2\in A$, a path of unitaries $\{w(t); t\in[0, 1]\}\subset A$ with $w(0)=1$ and $w(1)=w_1w_2w_1^*w_2^*=:w$, and a completely positive $\mathcal G'$-$\ep$-multiplicative linear maps $L_1, L_2: C\otimes\mathrm{C}(\mathbb T) \to A$  such that
\beq\label {pert-mea-cond-1}
&&\|[w_i, \phi(a)]\|<\ep \tforal a\in {\cal F},\,\,\,i=1, 2,\\\label
{pert-mea-cond0}
&&\|[w(t),\, u]\|<\ep,\,\,\, \|[w(t), \phi(a)]\|<\ep,\,\,\,\tforal a\in\F,\,\,\, \tand t\in[0, 1],\\\label{pert-mea-cond1}
&&\|L_1(a\otimes z)-(\phi(a)uw)\|<\ep,\quad \|L_1(a\otimes 1)-\phi(a)\|<\ep,\quad\tforal a\in\F,\\\label{pert-mea-cond2}  &&\|L_2(a\otimes z)-(\phi(a) w)\|<\ep,\quad \|L_2(a\otimes 1)-\phi(a)\|<\ep,\quad\tforal a\in\F,\\\label{pert-mea-cond3}
&&|\tau\circ L_1(g)-\tau\circ L_2(g)|<\ep,\quad\tforal g\in \mathcal H,\ \tforal \tau\in \mathrm{T}(A),
\eneq
and
$$\mu_{\tau\circ L_i}(B_a)>\Delta_0(a),\quad i=1, 2, \tforal \tau\in T(A)$$ and for any open ball $B_a$ of $X\times\mathbb T$ with radius $a>3\sqrt{2}\eta$.
\end{lem}

\begin{proof}

Let $\tilde{\mathcal H}\subseteq C(X\times\mathbb T)$ (in the place of $\mathcal G$) and $\tilde{\ep}>0$ (in the place of $\delta$) be the finite subset and constant of Lemma 3.4 of \cite{Lin-AU11} with respect to $\Delta_{00}(a\sqrt{2}/2)a\sqrt{2}/8$, $\eta$ and $\lambda_1=\lambda_2=1/2$. Regarding $C(X\times\mathbb T)$ as the center of $C\otimes C(\mathbb T)$, the subset $\tilde{\mathcal H}$ is inside $C\otimes C(\mathbb T)$.

Then without loss of generality, one may assume that $\tilde{\mathcal H}\subseteq \mathcal H$ and $\ep<\tilde{\ep}$, and one may also assume
$$\mathcal G'=\{f'_i\otimes z^{m_i};\ f'_i\in C, m_i\in \Z, i=1, ..., N\},$$
$$\mathcal H=\{f_i\otimes z^{n_i};\ f_i\in C, n_i\in \Z, i=1, ..., N\},$$ $1\in\mathcal F$ and $\|f_i\|,\, \|f_i'\|\le 1.$  Choose $M\in\mathbb N$ so that $|m_i|, |n_i|<M$ for any $i=1, ..., N$, and denote by $$\F_1=\{f_i', f_i;\ i=1, ..., N\}.$$

Let the natural number ${N_1}$ satisfies $${\eta\in[\frac{1}{2^{N_1+1}}, \frac{1}{2^{N_1}})}.$$
{For each $1\le j\le N_1,$ by a compactness argument, choosing ${\cal O}_j$
to be  a finite collection of  open balls of $X$ with radius $1/2^{j+2}$ which has the following property: for any open ball $O_a$ of $X$ with radius $a\in [1/2^{j+1}, 1/2^j),$ there is an open ball $O'\in {\cal O}_j$ such that $O'\subset O_a.$

Put ${\cal O}=\cup_{j=1}^{N_1} {\cal O}_j.$}
For each ${O'\in\mathcal O_j}$, fix a norm-one positive function $g$ such that the support of $g_{O'}$ is in ${O'}$, and is constant one if restricted to the open ball with the same center of ${O'}$ and with the radius ${\frac{1}{2^{j+3}}}$. Then $g_{O'}P$ is a central element of $C$.
Put ${\cal T}=\{g_{O'}P: O'\in {\cal O}\}.$

By Lemma \ref{pert-comm-lem}, for any $\min\{\Delta(\frac{1}{2^{N_1+3}})/2^{N_1+7}, \ep/2\}>\ep'>0$, there are $\delta'>0$ and a finite subset $\mathcal G\subseteq C$ such that for any C*-algebra $A$, any unitary $v\in A$ with $$\|[\phi(c), v]\|<\delta',\quad\forall c\in\mathcal G,$$  there exists a unital \morp\,   $L: C\otimes {\mathrm C}(\T)\to A$ with
$$\|L(f\otimes z^n)-\phi(f)v^n\|<\ep'<\ep/16,\quad\forall f\in\F\cup\F_1,\ -M\leq n\leq M.$$
By choosing $\ep'$ sufficiently small, the resulting map $L$ is  $\mathcal G'$-$\ep$-multiplicative. Without loss of generality, one may assume that $\delta'<\ep$.

One then asserts that $\delta:=\delta'/2$ and $\mathcal G$ satisfy the lemma. Let $\phi: C\to A$ be a homomorphism and $u\in A$ be a unitary with
$$\|[\phi(c), u]\|<\delta,\quad\forall c\in\mathcal G.$$

{Choose} an integer ${K\ge \max\{2^6\pi/\eta, \, 4(M+1)\}}$.
%
%
Since $A$ is tracially approximately divisible, for any $ \min\{ \Delta(\frac{1}{2^{N_1+3}})/2^{N_1+7}, \ep/32M\}>\ep''>0$ (which will be fixed later), there is a projection $p\in A$, a
unital \SCA\, $B\subset A$ with $B\cong
M_k(\C),$ with $1_{B}=p$ and $k\geq K$ such that
\begin{enumerate}
\item\label{tr-app-div-1} $\tau(1-p)<\ep''/16$ for any $\tau\in T(A)$,
\item\label{tr-app-div-2} $\phi(\F\cup\F_1\cup\mathcal G\cup\mathcal T)\subseteq_{\ep''} B'\cap A$ and $u\in_{\ep''} B'\cap A,$
\end{enumerate}
where $B'\cap A$ is the relative commutant of $B$ in $A.$
Let $w'\in B\cong M_k(\C)$ which has the following matrix form
\beq\label{La1-1}
w'=\begin{pmatrix}
e^{2\pi i/k}  & 0 & 0 &  \cdots\\
0 & e^{2\pi i 2 /k} & 0 &\cdots\\
&& \ddots \\
    0 &0& \cdots&        e^{2\pi i k/k}
              \end{pmatrix}.
\eneq
We compute that
\beq\label{La1-2}
t(w')=0,
\eneq
where $t\in T(B)$ is the tracial state. Moreover, for any $0<|n|\le M,$
\beq\label{La1-3}
t((w')^n)=1+\sum_{j=1}^{k-1} e^{2\pi nj i /k}={1-e^{2\pi nk i /k}\over{1-e^{2\pi n i /k}}}=0.
\eneq
In particular, $w' \in DU(B).$
Note that, since $B\cong M_k,$ there exist two unitaries
 $w'_1, w'_2\in B$  such that  $w'=w_1'w_2'(w_1')^*(w_2')^*.$
Let $\{w'(t); t\in[0, 1]\}\subseteq B$ be a continuous path of unitaries such that $w'(0)=1_{B }=p$ and $w'(1)=w'$. Denote by  $w_1=(1-p)+w_1'$, $w_2=(1-p)+w_2'$ and $$w=(1-p)+w'\quad\textrm{and}\quad w(t)=(1-p)+w'(t).$$

It is clear that (\ref{pert-mea-cond-1}) holds   when $\ep''$ sufficiently small.
By choosing $\ep''$ smaller, it follows from (2) above  that
$$\|[w(t),\, u]\|<\delta/2<\ep\andeqn \|[w(t), \phi(a)]\|<\delta/2<\ep,\quad\forall a\in\F, \forall t\in[0, 1],$$ which also verifies (\ref{pert-mea-cond0}).

One also assume that $\ep''$ is even sufficiently small so that for any $c\in\mathcal G$
\beq
\|[\phi(c), uw]\|<\delta',\quad \|[\phi(c), w]\|<\delta',\andeqn\, \|(uw)^n-u^nw^n\|<\ep/16,\quad -M\leq n\leq M.
\eneq
Then there are $\mathcal G'$-$\ep$-multiplicative linear maps $L_1, L_2: C\otimes\mathrm{C}(\mathbb T)\to A$ such that
$$\|L_1(f\otimes z^n)-\phi(f)(uw)^n\|<\ep'<\ep/16,\quad\forall f\in\F\cup\F_1,\ -M\leq n\leq M,$$ and
$$\|L_2(f\otimes z^n)-\phi(f)w^n\|<\ep'<\ep/16,\quad\forall f\in\F\cup\F_1,\ -M\leq n\leq M.$$

Since $1\in\mathcal F$, the maps $L_1$ and $L_2$ satisfy  (\ref{pert-mea-cond1}) and (\ref{pert-mea-cond2}).
Let us verify (\ref{pert-mea-cond3}). Let $\tau$ be any tracial state of $A$. Note that, for any $a\in B\subseteq pAp$ and any $b\in B'\cap pAp$, one has that $\tau(ba)=\tau(b)\tau(a)=\tau(b)\tau(p)t(a)$, where $t$ is the unique tracial state on $B$.

For any $i=1, ..., N$, choose $a_i', u''\in(1-p)A(1-p)$ and $a_i, u'\in B'\cap pAp,$ where $u', u''$ are unitaries and $\|a_i\|,\, \|a_i'\|\le 1$ such that $$\|(a_i+a_i')-\phi(f_i)\|<\ep''<\ep/32\quad \textrm{and}\quad \|(u'+u'')-u\|<\ep''<\ep/32M.$$  Then
\begin{eqnarray*}
\tau\circ L_1(f_i\otimes z^{n_i})&\approx_{\ep/16}&\tau(\phi(f_i)(uw)^{n_i})\\
&\approx_{\ep/16}&\tau(\phi(f_i)u^{n_i}w^{n_i})\\
&\approx_{\ep/16}& \tau(a_i(u')^{n_i} w^{n_i})\\
&=&\tau(a_i(u')^{n_i}) t(w^{n_i})\\
&\approx_{\ep/16}&\left\{\begin{array}{ll}
\tau(\phi(f_i)) & \textrm{if $n_i=0$}\\
0 & \textrm{if $n_i\neq0$}
\end{array}
\right. ,
\end{eqnarray*}
and
\begin{eqnarray*}
\tau\circ L_2(f_i\otimes z^{n_i})&\approx_{\ep/16}&\tau(\phi(f_i)w^{n_i})\\
&\approx_{\ep/16}& \tau(f_i' w^{n_i})\\
&=&\tau(a_i) t(w^{n_i})\\
&\approx_{\ep/16}&\left\{\begin{array}{ll}
\tau(\phi(f_i)) & \textrm{if $n_i=0$}\\
0 & \textrm{if $n_i\neq0$}
\end{array}
\right. .
\end{eqnarray*}
Thus, $$|\tau\circ L_1(f_i\otimes z^{n_i})-\tau\circ L_2(f_i\otimes z^{n_i})|<\ep.$$
Note that we choose
$K\ge \max\{2^6\pi/\eta, \, 4(M+1)\}.$ In particular,
$$
{2\pi\over{K}}\le 1/2^{N_1+5}.
$$
One then computes that
$$|\mu_t(I)-|I||<\frac{1}{2^{N_1+3}}$$ for any arc $I$ with length at least $\eta$, where $\mu_t$ is the Borel probability measure induced  by positive linear functional $t\circ f(w)$ on $C(\T),$ where $t$ is the tracial state on $B.$

Now, let ${B_a}$ be any open ball on $X\times\mathbb T$ with radius
$a$. Denote by $(a_0, b_0)$ the {center} of $B_a$.
Denote by ${O_{a\sqrt{2}/2}}$ the open ball of $X$ with radius
$a\sqrt{2}/2$ and {center} $a_0$, and denote by ${J_{a\sqrt{2}/2}}$ the open
ball of $\mathbb T$ with radius $a\sqrt{2}/2$ and { center} $b_0$. Note
that ${ O_{a\sqrt{2}/2}\times J_{a\sqrt{2}/2}\subseteq O_a}$.

Assume that $a\sqrt{2}/2\in[\frac{1}{2^{j+1}}, \frac{1}{2^j})$ {for some
 $1\le j\le N_1$} Then choose $O'_j\in\mathcal O_j$ such that $O'_j\subseteq O'_{a\sqrt{2}/2}$, and consider $g_1P\in\mathcal T$ associated to $O'_k$, and any norm-one positive continuous function $g_2$ on $\mathbb T$ with support in $J_{a\sqrt{2}/2}$. Note that $$\phi(g_1)\approx_{\ep''} a+b$$ for some $b$ commutes with $B$ and the traces of $a$ are at most $\ep''$.

Consider the function $g(x, t):=g_1(x)P\cdot g_2(t).$
Then, for any $a\geq\sqrt{2}\eta>\eta$,
\begin{eqnarray}
\mu_{\tau\circ L_2}(B_a) \geq \tau(L_2(g))&>&\tau(b g_2(w))-\ep'\\
&=&\tau(b)\cdot t(g_2(w))-\ep'\\
&>& \tau(\phi(g_1))\cdot t(g_2(w))-\ep'-\ep''\\
&>&\Delta(\frac{1}{2^{k+3}})\cdot t(g_2(w))-\ep'-\ep''
\end{eqnarray}
Since this holds for any $g_2$, one has
$$\mu_{\tau\circ L_2}(B_a)\geq \Delta(\frac{1}{2^{k+3}})\cdot \mu_t(J_{a\sqrt{2}/2})/2-\Delta(\frac{1}{2^{N_1+3}})/2^{N_1+5}=\Delta(\frac{1}{2^{k+3}})a\sqrt{2}/8>\Delta_{00}(a\sqrt{2}/2)a\sqrt{2}/8,$$
where $\mu_t$ is the spectral measure of $w$.

Note that $$|\tau\circ L_1(g)-\tau\circ L_2(g)|<\ep<\tilde{\ep}\quad\forall g\in\tilde{\mathcal H}\subseteq\mathcal H,$$ by Lemma 3.4 of \cite{Lin-AU11}, one has $$\mu_{\tau\circ L_1}(B_a)>\frac{\sqrt{2}}{48}\Delta_{00}(a\sqrt{2}/6)a=\Delta_0(a)$$ for any $a\geq 3\sqrt{2}\eta$.
\end{proof}

\begin{df}
{\rm Let $A$ be a unital \CA. In the following, for any invertible element $x\in A$, let $\langle x \rangle$ denote the unitary $x(x^*x)^{-\frac{1}{2}}$, and let $\overline{x}$ denote the element $\overline{\langle x \rangle}$ in $U(A)/CU(A)$. Consider a subgroup $\Z^k\subseteq K_1(A)$, and write the unitaries $\{u_1, ..., u_k\}\subseteq U_c(A)$ corresponding to the standard generators $\{e_1, e_2,...,e_k\}$ of $\Z^k$. Suppose that $\{u_1, u_2,...,u_k\}\subset M_n(A)$ for some integer $n\ge 1$. Let $\Phi: A\to B$ be a  unital positive linear map such that $\Phi\otimes {\mathrm{id}}_{M_n}$ is
at least $\{u_1, ..., u_k\}$-$1/4$-multiplicative (hence each $\Phi\otimes {\mathrm{id}}_{M_n}(u_i)$ is invertible), then the map $\Phi^\ddag|_{s_1(\Z^k)}: \Z^k\to U(B)/CU(B)$ is defined by
$$\Phi^\ddag|_{s_1(\Z^k)}(e_i)= \overline{\langle\Phi\otimes {\mathrm{id}}_{M_n}(u_i)\rangle},\quad 1\leq i\leq k.$$
Thus, for any finitely generated subgroup
$G\subset {\overline{U_c(A)}},$ there exists $\dt>0$ and a finite subset ${\cal G}\subset A$ such that, for any unital $\dt$-${\cal G}$-multiplicative \morp\,
$L: A\to  B$ (for any unital \CA\, $B$), the map $L^{\ddag}$ is well defined
on  $s_1(G).$ (Please see  \ref{DDet} for $U_c(A)$ and $s_1.$)
}
\end{df}

\begin{thm}\label{hmtp-H}
Let $C=\mathrm{C}(X)$ with $X$ a compact
metric space
and let $\Delta: (0, 1) \to (0, 1)$ be a non-decreasing map. For any $\ep>0$ and any finite subset $\mathcal F \subseteq C$, there exists $\delta>0$, $\eta>0$, $\gamma>0$, a finite subsets $\mathcal G \subseteq C$, $\mathcal P\subseteq \underline{K}(C)$,  a finite
 subset ${\cal Q}=\{x_1, x_2,...,x_m\}\subset K_0(C)$ which generates a free
 subgroup and $x_i=[p_i]-[q_i],$ where $p_i, q_i\in M_n(C)$
 (for some integer $n\ge 1$) are projections,
 satisfying the following:

Suppose that $A$  is a unital simple C*-algebra with $TR(A)\leq 1$, $\phi: C\to A$ is a unital homomorphism and $u\in A$ is a unitary, and suppose that
$$\|[\phi(c), u]\|<\delta,\ \forall c\in\mathcal G\quad\mathrm{and}\quad \mathrm{Bott}(\phi, u)|_{\mathcal P}=0,$$ and
$$\mu_{\tau\circ \phi}(O_a)\geq\Delta(a)\ \forall \tau\in T(A),$$
where $O_a$ is any open ball in $X$ with radius $\eta\leq a<1$ and $\mu_{\tau\circ \phi}$ is the Borel probability measure defined
by $\tau\circ \phi$. Moreover, for each $1\leq i\leq m$,
there is $v_i\in CU(M_n(A))$ such that
$$
\|\langle(\mathbf 1_n-\phi(p_i)+\phi(p_i){1}_n\otimes u)
(\mathbf 1_n-\phi(q_i)+\phi(q_i)({1}_n\otimes u^*\rangle-v_i\|<\gamma.
$$
Then there is a continuous path of unitaries $\{u(t): t\in[0, 1]\}$ in $A$ such that
$$u(0)=u, u(1)=1,\ \textrm{and}\ \|[\phi(c), u(t)]\|<\ep$$
for any $c\in\mathcal F$ and for any $t\in[0, 1]$.
\end{thm}

\begin{proof}
Since $A$ is a simple \CA\, with $TR(A)\le 1,$ it is tracially
approximately divisible (see 5.4 of \cite{LinTAI}). Therefore \ref{pert-meas-lem} applies.
Without loss of generality, one may assume that $\F$ is in the unit ball of $C$. Let $\ep_0$ be the universal constant such that, for any unitaries $u_1$ and $u_2$ with $\|u_1-u_2\|<\ep_0$, there is a path of unitaries connecting $u_1$ and $u_2$ with length at most $\ep/2$.

Let $\eta'>0$, $\delta'>0$, $\mathcal G'\subseteq C\otimes \mathrm{C}(\mathbb T)$, $\mathcal H\subseteq C\otimes \mathrm{C}(\mathbb T)$, $\mathcal P'\subseteq\underline{K}(C\otimes\mathrm{C}(\mathbb T))$, $\mathcal U'\subseteq U_c(K_1(C\otimes \mathrm{C}(\mathbb T)))$, $\gamma_1$, and $\gamma_2$ be the finite subsets and constants of Theorem 5.3 of \cite{Lin-AU11} with respect to $X\times\mathbb T$, $\Delta_0$, $\F\otimes\{1, z\}$, and $\min\{\ep/2, \ep_0\}.$ Without loss of generality, we may assume
that
$
{\cal P}'={\cal P}\cup {\boldsymbol{\bt}({\cal P})},
$
where ${\cal P}$ is a finite subset of $\underline{K}(C),$ and
$$
{\cal G}'={\cal G}_1'\cup \{1_{\mathrm{C}(\T)}, z\},
$$
where
${\cal G}_1'$ is a finite subset of $C.$
Moreover, we may assume
 \beq\label{hmtp-H-add}
 [L']|_{\cal P}=[L'']|_{\cal P}
 \eneq
 for any unital ${\cal G}'_1$-$\dt'$-multiplicative \morp s $L',L'': C\to A$ with
$$
\|L'(g)-L''(g)\|<\gamma_2\tforal g\in {\cal G}_1'.
$$
By choosing larger ${\cal G}_1'$ and smaller $\dt',$
we may assume further that $(L')^{\ddag}$ is well defined
on $\overline{{\cal U}'}.$

Since $K_1(C\otimes \mathrm{C}(\mathbb T))=K_1(C)\oplus K_0(C)$, without of loss of generality,
the set $\mathcal U'$ may be chosen as $\mathcal U'_1\cup\mathcal U'_0$,
where $\mathcal U'_1=\{\overline{v_1\otimes 1_{\mathrm{C}(\T)}}, ..., \overline{v_{n'}\otimes 1_{\mathrm{C}(\T)}}\}$ with each
$v_i$ a unitary $M_n(C)$, and any element in $\mathcal U_0'$ has the form
$$\overline{(p\otimes z+(\mathbf{1}_n-p)\otimes 1_{\mathrm{C}(\T)})(q\otimes z+(\mathbf{1}_n-q)\otimes 1_{\mathrm{C}(\T)})^*}$$ for some projections $p$ and $q$ in $M_n(C)$ for some integer $n\geq 1$. Without loss of generality, one may assume that $\mathcal U_0'$ exactly generates a free group $\Z^m$ in $K_1(C\otimes \mathrm{C}({\mathbb T}))$ as standard generators, and hence one may write
$$\mathcal{U}_0'=\{\overline{(p_{i}\otimes z+(\mathbf{1}_n-p_{i})\otimes 1_{\mathrm{C}(\T)})(q_{i}\otimes z+(\mathbf{1}_n-q_{i})\otimes 1_{\mathrm{C}(\T)})^*};\ i=1, ..., m\},$$
where $p_i$ and $q_i$ are
 projections in $M_n(C)$.
Denote by $x_i=[p_i]-[q_i]$ for $1\leq i\leq m$, and put $\mathcal Q=\{x_1, ..., x_m\}.$

We may assume that ${\cal F}_1\subset C$ is a finite subset
such that
$$
p_i, q_i\in \{(c_{j,k})\in M_n(C): c_{j,k}\in {\cal F}_1\}.
$$
Put ${\cal F}_2= \{1_C\}\cup {\cal F}\cup {\cal F}_1.$
Let $\delta>0$ and $\mathcal G\subseteq C$ be the constant and finite subset of Lemma \ref{pert-meas-lem} with respect to $\min\{\ep/8n^2, \delta'/n^2, \gamma_1/2n^2, \gamma_2/16n^2\}$ (in place of $\ep$), ${\mathcal F_2}$ (in place of ${\cal F}$),  $\mathcal G'$ and
$\mathcal H.$

Without loss of generality, one may assume that $\delta$ is sufficiently small and $\mathcal G$ is sufficiently  large such that
$\mathrm{Bott}(\phi, u_1u_2)|_{\mathcal P}$ is well defined and
$$\mathrm{Bott}(\phi, u_1u_2)|_{\mathcal P}=\mathrm{Bott}(\phi, u_1)|_{\mathcal P}+\mathrm{Bott}(\phi, u_2)|_{\mathcal P}
$$
for any unital homomorphisms $\phi: C\to B$ for some unital C*-algebra $B$ and  unitaries $u_1, u_2\in B$ with
$$\|[\phi(a), u_i]\|<\delta,\quad \forall a\in\mathcal G, i=1, 2.$$

One asserts that $\delta$, $\eta=\frac{\sqrt{2}}{6}\eta'$, $\gamma=\gamma_2/4$, $\mathcal P$, $\mathcal G$ and $\mathcal Q$ satisfy the theorem.

Let $(\phi, u)$ be a pair which satisfies the condition of the theorem.
By Lemma \ref{pert-meas-lem}, there are unitary $w=w_1w_2w_1^*w^*_2$ with $w_1, w_2$ unitaries in $A$, a path of unitaries $\{w'(t); t\in[0, 1]\}$ with $w'(1)=1$ and $w'(0)=w$, and unital $\mathcal G'$-$\delta'$-multiplicative completely positive linear maps $L_1, L_2: C\otimes\mathrm{C}(\mathbb T) \to A$ such that for any $f\in\mathcal F$,
\begin{enumerate}
\item\label{pert-mea-cond-1-01} $\|[w_i, \phi(a)]\|<\min\{\ep/8n^2, \gamma_2/16n^2\}$, $\forall a\in\F_2$, i=1, 2,
\item\label{pert-mea-cond0-01} $\|[w'(t),\, u]\|,\,\|[w'(t), \phi(a)]\|<\min\{\ep/8n^2, \gamma_2/8n^2\}$, $\forall a\in\F_2$, $\forall t\in[0, 1]$,
\item\label{pert-mea-cond1-01} $\|L_1(f\otimes z)-(\phi(f)uw)\|
<\min\{\ep/8n^2, \gamma_2/8n^2\},\quad
\|L_1(f\otimes 1)-\phi(f)\|<\min\{\ep/8n^2, \gamma_2/8n^2\},$
\item\label{pert-mea-cond2-01}  $\|L_2(f\otimes z)-(\phi(f) w)\|<\min\{\ep/8n^2, \gamma_2/8n^2\},\quad \|L_2(f\otimes 1)-\phi(f)\|<\min\{\ep/8n^2, \gamma_2/8n^2\},$
\item\label{pert-mea-cond3-01}  $|\tau\circ L_1(g)-\tau\circ L_2(g)|<\gamma_1/2n^2,\quad\forall g\in \mathcal H,\ \forall \tau\in \mathrm{T}(A)$,
\item\label{pert-mea-cond4-01} $\mu_{\tau\circ L_i}(O_a)>\Delta_0(a)$, $i=0, 1$ for any open ball $O_a$ of $X\times\mathbb T$ with radius $a>3\sqrt{2}\eta=\eta',$
\end{enumerate}

It follows from (\ref{pert-mea-cond0-01})  that
 $\mathrm{Bott}(\phi, w)=0.$ 
 Therefore
\begin{eqnarray*}
\mathrm{Bott}(\phi, uw)|_{\mathcal P}&=&\mathrm{Bott}(\phi, u)|_{\mathcal P}+\mathrm{Bott}(\phi, w)|_{\mathcal P}\\
&=&\mathrm{Bott}(\phi, u)|_{\mathcal P}=0.
\end{eqnarray*}
We also have, by (\ref{hmtp-H-add}),
\begin{equation}\label{alg-kk-01}
[L_1]|_{\mathcal P}=[\phi]|_{\mathcal P}=[L_2]|_{\mathcal P}.
\end{equation}

Note that, by (\ref{pert-mea-cond-1-01}), $$w=w_1w_2w_1^*w_2^*$$ with $\|[w_i, \phi(a)]\|<\min\{\ep/8n^2, \gamma_2/16n^2\}$, $\forall a\in\F_2$, $i=1, 2$. Then for any projection $p_i$ (or $q_i$), one estimates  that
\beq\label{Na-1}
{\rm dist}(\langle (\mathbf1_n-\phi(p_i))+\phi(p_i) {\tilde w}\rangle, CU(M_n(A)))&<&\gamma_2/16\andeqn\\
 {\rm dist}(\langle (\mathbf1_n-\phi(q_i))+\phi(q_i){\tilde  w}\rangle, CU(M_n(A)))&<& \gamma_2/16,
\eneq
for any $1\le i\le m,$ where ${\tilde w}={\rm diag}(\overbrace{w,w,...w}^n).$
Therefore, for any $1\leq i\leq m$,
\beq
&&\hspace{-0.6in}\mathrm{dist}(\overline{L_2((p_i\otimes z+{\mathbf 1}_n-p_i)(q_i\otimes z+{\mathbf 1}_n-q_i)^*)}, \overline{\mathbf 1_n})\approx_{\gamma_2/4}0,\andeqn\\
&&\hspace{-0.6in}\mathrm{dist}(\overline{\langle L_1((p_i\otimes z+({\mathbf 1}_n-p_i))(q_i\otimes z+({\mathbf 1}_n-q_i))^*)\rangle}, \overline{\mathbf 1_n})\\
&\approx_{\gamma_2/8}& \mathrm{dist}(\overline{\langle((\mathbf {1}_n- \phi(p_i))+\phi(p_i){{\widetilde{u w}}})((\mathbf {1}_n-q_i)+\phi(q_i){{\widetilde{u w}}})^*\rangle}, \overline{\mathbf 1_n})\\
&\approx_{\gamma_2/8}& \mathrm{dist}(\overline{\langle((\mathbf {1}_n- \phi(p_i))+\phi(p_i){{\tilde u}})((\mathbf {1}_n-q_i)+\phi(q_i){{\tilde u}})^*\rangle}, \overline{\mathbf 1_n})\\
&<&\gamma=\gamma_2/4,
\eneq
{where ${\tilde a}={\rm diag}(\overbrace{a,a,...,a}^n).$}
Also note that for any $v_i\otimes 1\in \mathcal U'_1$, one computes that
$$\mathrm{dist}(\overline{\langle L_1(v_i\otimes 1)\rangle}, \overline{L_2(v_{i}\otimes 1)})\approx_{\gamma_2}\mathrm{dist}(\overline{\phi(v_i)}, \overline{\phi(v_{i})})=0.$$
Since $U_0(A)/CU(A)$ is torsion free (Theorem 6.11 of \cite{LinTAI}), one has that
\begin{equation}\label{alg-k1-01}
\mathrm{dist}(\overline{\langle L_1(u)\rangle}, \overline{\langle L_2(u)\rangle })<\gamma_2,\quad\forall u\in \mathcal U'.
\end{equation}
By \eqref{alg-kk-01} \eqref{pert-mea-cond3-01}, \eqref{alg-k1-01} and \eqref{pert-mea-cond4-01}, it follows from Theorem 5.3 of \cite{Lin-AU11} that there is a unitary $U\in A$ such that
$$\|L_1(f)-U^*L_2(f)U\|<\min\{\ep/2, \ep_0\},\quad\forall f\in\F\otimes\{1, z\}.$$

Consider the path of unitaries $w(t): t\mapsto U^*w'(2t-1)U$, $t\in[1/2, 1].$ Then
\beq
\|[\phi(f), w(t)]\|<\ep\quad\forall f\in \mathcal F, t\in[1/2, 1]\andeqn
\|w(1/2)-uw\|<\ep_0, \quad w(1)=1.
 \eneq
 By the choice of $\ep_0$, there is a path of unitaries $\{w''(t);\ t\in[1/4, 1/2]\}$ such that
 \beq
 \|[\phi(f), w(t)]\|<\ep,\quad\forall f\in\mathcal F,\ t\in[1/4, 1/2],\andeqn\\
w''(1/4)=uw\quad\mathrm{and}\quad w''(1/2)=w(1/2).
\eneq
Also consider the path of unitaries $w'''(t): t\to uw'(4t)$, $t\in[0, 1/4].$ Then one has that $w'''(0)=u$, $w'''(1/4)=uw$ and $$\|[w'''(t), \phi(f)]\|<\ep,\quad\forall f\in C.$$ Define the path
$$
w(t)=\left\{
\begin{array}{ll}
w'''(t), & \textrm{if $t\in[0, 1/4]$},\\
w''(t), & \textrm{if $t\in [1/4, 1/2]$},\\
w(t), &\textrm{if $t\in[1/2, 1]$}.
\end{array}
\right.
$$
Then it is clear that
$$\|[\phi(f), w(t)]\|<\ep,\quad\forall f\in \F, t\in[0, 1],$$ $$w(0)=u\quad\mathrm{and}\quad w(1)=1,$$
as desired.
\end{proof}

\begin{cor}\label{hp-mat}
Let $X$ be a {compact subset of finite CW-complex} and let $C=PM_n(\mathrm{C}(X))P$ for some integer $n\geq 1$ and $P$ a projection in $M_n(\mathrm{C}(X))$. Then the statement of Theorem \ref{hmtp-H} still holds for the C*-algebra $C$ and using the measure define in \ref{mea}.
\end{cor}
\begin{proof}

If $C=M_n(C(X))$, it is clear that the corollary follows from Theorem \ref{hmtp-H} ($X$ is even not required to have finite covering space in this case). In what follows we will use this case of the corollary to prove the general case.

Assume that $C=PM_n(C(X))P$. {Since $X$ is compact, the rank of $P$ has only finitely many values. It follows that, without loss of generality,  we may assume} that $P(x)\not=0$ for all $x\in  X.$ Since $X$ {is a compact subset of finite CW-complex,} there is an integer $d$ and a projection $Q\in M_d(PM_n(C(X))P)$ such that
$$QM_d(PM_n(C(X))P)Q\cong M_{r}(C(X))$$
for some integer $r$. Note that
$Q(x)\not=0$ for all $x\in X.$ Without loss of generality, one may assume that $P \preceq Q$, that is, there is also a partial isometry $V\in M_d(PM_n(C(X))P)$ such that $VV^*\leq Q$ and $V^*V=\mathrm\{P, 0, ..., 0\}$.

There is an integer $l\geq 1$ such that $X=X_1\sqcup\cdots\sqcup X_l$ such that the ranks of the restrictions of $P$ and $Q$ to each $X_i$, $1\leq j\leq l$, are constant. Denote by $P_j$ and $Q_j$ the restriction of $P$ and $Q$ to $X_j$ respectively.
Let $R_1=\max_{1\le j\le l}\{{\rm rank} P_j\}$ and $R_2=\min_{1\le j\le l}\{{\rm rank} Q_j\}.$

Fix $d$, $Q$, and $V$.
Let $\Delta: (0, 1) \to (0, 1)$ be a non-decreasing map, let $\ep>0$ and $\F\subseteq PM_n(C(X))P$ be a finite subset of elements with norm one.

Pick $\frac{\ep}{4}>\ep'>0$ such that for any unitaries $u, v$ in a C*-algebra with $\|u-v\|<\ep'$, there is a path of unitaries $u(t)$ such that $u(0)=u$, $u(1)=v$, and $\|u(t)-v\|<\frac{\ep}{2}$, $\forall t\in[0, 1]$.

Pick $\frac{\ep'}{4}>\ep''>0$ such that if there are a projection $p$ and a unitary $U$ in a C*-algebra $A$ with $\|[p, U]\|<\ep''$, then
$$\|\langle pUp\rangle-pUp\|<\ep'/4.$$
(Recall that $\langle pUp\rangle=pUp(pU^*pUp)^{-\frac{1}{2}}$.)

Denote by  $\dt'$ (in place $\dt$), $\eta,$ $\gamma'$ (in place of $\gamma$),
${\cal G}'\subseteq QM_d(P(M_n(C(X)))P)Q\cong M_{r}(C(X))$ (in place of ${\cal G}$), $\mathcal P\subseteq \underline{K}(C(X))$, and $\mathcal Q\subseteq K_0(C(X))$ the constants and finite subsets of the corollary required for $M_{r}(C(X))$ with $\ep''$, $V \F V^*$, and $\Delta$.

We may assume that $\gamma'<1.$
For each $x_i\in\mathcal Q$, write $x_i=[p_i]-[q_i]$ with $p_i, q_i\in M_k(QM_d(C)Q)$ for some integer $k$. Choose an integer $k'$ such that $$M_k(QM_d(C)Q)\subseteq M_{k'}(C).$$
Without loss of generality, one also assumes that any element of $ {\cal G}'$ has norm one, and $VV^* \in \mathcal G'$.
Choose a finite subset $\mathcal G_1\subseteq C$ and $\delta_1>0$ such that if there is a C*-algebra $A$ and a unitary $u\in A$ satisfies $$\|[\phi(c), u]\|<\delta_1,\quad \forall c\in\mathcal G_1$$ for some homomorphism $\phi$ to $A$, then
$$\|[(\phi\otimes {\rm {id}}_{M_d})(c), u\otimes 1_{M_d}]\|<\delta'/2$$ for any $c\in\mathcal G'\subseteq M_d(C)$, and
$$\|[\phi\otimes {\rm id}_{M_d}(Q), u\otimes 1_{M_d}]\|<\min\{\ep'', \delta'/2\}.$$

Let $B=QM_d(C)Q\otimes C(\T).$ It is a full hereditary \SCA\, of $M_d(C)\otimes C(\T).$
Choose a large finite subset ${\cal G}_2\subset C$ and a sufficiently small
$\dt_2>0$  such that, if $L: M_d(C)\otimes C(\T)\to M_d(A)$ is a unital ${\cal G}_2\times \{1, z\}$-$\dt_2$-multiplicative \morp\, and $[L]|_{\cal P}$ is well defined, then
$$
[L]|_{\cal P}=
[L|_B]|_{{\cal P}}.
$$
Note that if we assume that
\beq
{\rm Bott}(\phi, u)|_{\cal P}=0,
\eneq
then
\beq
{\rm Bott}(\phi\otimes {\rm id}_{M_d}, u\otimes 1_{M_d})|_{\cal P}=0.
\eneq
It then follows that we can choose a  larger ${\cal G}_2$  and smaller
$\dt_2$  so that
if
\begin{equation}\label{approx-comm-2}
\|[\phi(c), u]\|<\delta_2,\quad\forall c\in\mathcal G_2\andeqn
\mathrm{Bott}(\phi, u)|_\mathcal P=0,
\end{equation}
 we still have
\beq\label{AddBott}
{\rm Bott}(\phi|_{QM_d(C)Q}, \langle q(u\otimes 1_{M_d})q\rangle )|_{\cal P}=0,
\eneq
where $q=(\phi\otimes {\rm id}_{M_d})(Q).$

Note that $p_i, q_i\in M_k(QM_d(C)Q)\subseteq M_{k'}(C)$. Define ${\tilde{q}}=q\otimes 1_{M_k}$. Then there is a finite subset $\mathcal G_3\subseteq C$, and $\delta_3>0$ such that if
\beq\label{approx-comm-3}
&&\|[\phi(c), u]\|<\delta_3,\quad\forall c\in\mathcal G_3
\andeqn\\\nonumber
&&\hspace{-0.6in}\|\langle (1_{M_{k'}}-\phi(p_i)+\phi(p_i)1_{M_{k'}}\otimes u) (1_{M_{k'}}-\phi(q_i)+\phi(q_i)1_{M_{k'}}\otimes u^*)) \rangle-v_i\|<\gamma'/(8(k'R_1+\frac{1}{8})),
\eneq
for some $v_i\in CU(M_{k'}(A))$,
then
\begin{equation}
\|g_i-((1_{M_{k'}}-{\tilde{q}})+\langle {\tilde{q}}g_i{\tilde{q}}\rangle )\|<\gamma'/(4(k'R_1+\frac{1}{8})),
\end{equation} and
\begin{equation}
\|{g}'_i-\langle {\tilde{q}}g_i {\tilde{q}})\rangle\|<\gamma'/(4(k'R_1+\frac{1}{8})),
\end{equation}
where
$$g_i:=\langle (1_{M_{k'}}-\phi(p_i)+\phi(p_i)  u\otimes 1_{M_{k'}} ) (1_{M_{k'}}-\phi(q_i)+\phi(q_i) u^*\otimes 1_{M_{k'}}) \rangle$$
and
$${g'_i}:=\langle ({\tilde{q}}-\phi(p_i)+\phi(p_i) \langle {\tilde{q}} (u\otimes 1_{M_{k'}}){\tilde{q}}\rangle ) ({\tilde{q}}-\phi(q_i)+\phi(q_i) \langle {\tilde{q}}(u^*\otimes 1_{M_{k'}}){\tilde{q}} \rangle) \rangle.$$
Note that, in particular, one has
\begin{equation}
\|g_i-((1_{M_{k'}}-{\tilde{q}})+g'_i)\|<\gamma'/{2(k'R_1+\frac{1}{8})}.
\end{equation}
Then
\begin{equation}\label{n6/21}
\mathrm{dist}((1_{M_{k'}}-{\tilde{q}})+g'_i, CU(M_{k'}(A)))<\gamma'/{(k'R_1+\frac{1}{8})}.
\end{equation}
Since $TR(A)\le 1,$  it follows from Lemma 6.9 of \cite{LinTAI} that
$CU(M_{k'}(A))\subset U_0(M_{k'}(A)).$
It follows from the fact that $\gamma'<1$ and (\ref{n6/21}) that
$(1_{M_{k'}}-{\tilde{q}})+g'_i\in U_0(M_{k'}(A)).$ Since $A$ is a unital simple \CA\, with $\tr(A)\le 1,$  it has
stable rank one (Theorem 4.5 of \cite{LinTAI}). Therefore one has that $g_i'\in U_0(M_k(qM_d(A)q))$ (see 2.10 of \cite{Rief-JOT}).
Note that for any $\tau\in T(M_{k'}(A))$, one has
$$\tau({\tilde{q}})\geq \frac{k R_2}{k' R_1}>\frac{1}{k' R_1},$$  and hence
$$k'R_1[{\tilde{q}}]>[1_{M_k'}-{\tilde{q}}].$$
Then by Lemma 3.3 of \cite{Lin-hmtp}, one has
$$\mathrm{dist}(g'_i, CU(M_k(qM_d(A)q)))<(k'R_1+\frac{1}{8})\frac{\gamma'}{(k'R_1+\frac{1}{8})}=\gamma'.$$
That is,
\begin{equation}\label{AddBu}
\|\langle ({\tilde{q}}-\phi(p_i)+\phi(p_i) \langle {\tilde{q}} (u\otimes 1_{M_d}){\tilde{q}} \rangle ) ({\tilde{q}}-\phi(q_i)+\phi(q_i) \langle {\tilde{q}}(u^*\otimes 1_{M_d}){\tilde{q}} \rangle) \rangle-v'_i\|<\gamma',
\end{equation}
for some $v'_i\in CU(M_k(qM_d(A)q))$.

Put $\gamma=\gamma'/(8(k'R_1+\frac{1}{8})).$
Then, one asserts that $\dt=\min\{\delta_1, \delta_2, \delta_3\}$, $\eta$, $\gamma$, $\mathcal G_1\cup\mathcal G_2\cup\mathcal G_3$, $\mathcal P$, and $\mathcal Q$ satisfy the corollary.

Let $\phi: PM_n(C(X))P\to A$ be a unital homomorphism satisfies the conditions of the corollary for some unitary $u\in A$, where $A$ is a simple C*-algebra with $TR(A)\leq 1$.

Put $v=\phi\otimes 1_{M_d}(V)\in M_d(A)$. The restriction of $\phi\otimes 1_{M_d}$ to $QM_d(C)Q$ (which is isomorphic to $M_{r}(C(X))$) is a unital homomorphism to $qM_d(A)q$, which has $TR(qM_d(A)q)\leq 1$, and one also has that $vv^*\leq q$ and $v^*v=1_A$.

Since
$\|[\phi(c), u]\|<\delta_1$, $\forall c\in\mathcal G_1,$
one has  $$\|[(\phi\otimes 1_{M_d})(c), u\otimes 1_{M_d}]\|<\delta'/2,\quad\forall c\in\mathcal G'.$$ In particular,
$$\|[(\phi\otimes 1_{M_d})|_{QM_d(C)Q}(c), \langle q(u\otimes 1_{M_d})q \rangle]\|<\delta',\quad\forall c\in\mathcal G'.$$ 

Since $\phi$ also satisfies \eqref{approx-comm-2} (and \eqref{approx-comm-3}), Equations \eqref{AddBott} (and \eqref{AddBu}) are also satisfied.

Since $\mu_{\tau\circ\phi}(O_a)\geq\Delta(a)$ for any open ball $O_a$ on $X$ with radius $1>a>\eta$ and any $\tau\in T(A)$, one then also has that $$\mu_{\tau\circ((\phi\otimes 1_{M_d})|_{QM_d(C)Q})}(O_a)\geq\Delta(a)$$ for any open ball $O_a$ on $X$ with radius $1>a>\eta$ and any tracial state $\tau$ on $qM_d(A)q$.

Then, applying the corollary to $QM_d(C)Q$ and $qM_d(A)q$, there is a path of unitaries $\{U(t);\ t\in[0, 1]\}\subseteq qM_d(A)q$ such that $$U(0)=1_{qM_d(A)q},\quad U(1)=\langle q(u\otimes 1_{M_d})q \rangle, $$ and
$$\|[\phi\otimes 1_{M_d}(VfV^*), U(t)]\|<\ep'',\quad\forall f\in \F.$$

Denote by $e=vv^*\in qM_d(A)q$. Note that $\|[e, U(t)]\|<\ep''<\frac{1}{4}$. One considers the path of unitaries
$$w(t)= \langle eU(t)e\rangle \in eM_d(A)e,\quad t\in[0, 1].$$
Then $$w(0)=r,\quad \|w(1)-e(u\otimes 1_{M_d})e\|<\ep'/2,$$
$$\|[v(\phi\otimes 1_{M_d})(f)v^*, w(t)]\|<2\ep'+2\ep''<\ep,\quad\forall f\in\F.$$

Consider the path of unitaries $u(t):=v^*w(t)v\in A.$ One then has that $$u(0)=1_A,\quad \|u(1)-u\|<\ep'/2+\ep''<\ep',$$ and
$$\|[\phi(f), u(t)]\|<\ep,\quad\forall f\in\F.$$
%
\end{proof}

\begin{rem}
In fact, the corollary above holds for the case that $X$ is
a general compact metric space. One can use a standard argument reducing
the general case to the case that $X$ is a compact subset of a finite CW-complex.
\end{rem}

The following lemma is due to N.C. Phillips. (See the proof of 3.8 of \cite{Phil-AJM}.)
\begin{lem}\label{lem64}
Let $A$ be a unital C*-algebra and $2>d>0$. Let $u_0, u_1, ..., u_n$ be $n+1$ unitaries in $A$ such that $$u_n=1\quad\textrm{and}\quad\|u_i-u_{i+1}\|\leq d\quad i=0, 1,..., n-1.$$
Then there exists a unitary $v\in CU(\mathrm{M}_{2n+1}(A))$ with exponential length no more than $2\pi$ such that $$\|(u_0\oplus1_{\mathrm{M}_{2n}(A)})-v\|
\leq d.$$
\end{lem}

In the rest of the paper, unless otherwise specified, $z$ will be the identity function
on the unit circle.

\begin{thm}\label{BB-exi}
Let $C=\mathrm{C}(X)$ with $X$ a compact metric space, let
 $G\subset K_0(C)$ be a finitely generated
 subgroup. Write
 ${G}=\Z^k\oplus {\mathrm{Tor}}({G})$ with
 $\Z^k$ generated by
$$
\{x_{1}=[p_1]-[q_1], x_2=[p_2]-[q_2], ..., x_{k}=[p_{k}]-[q_{k}]\},
$$
where $p_i, q_i\in M_n({\mathrm C}(X))$ (for some integer $n\ge 1$) are projections,
$i=1,...,k$.

Let $A$ be a simple C*-algebra with $TR(A)\leq 1$. Suppose that $\phi: C\to A$ is a monomorphism. Then, for any finite subsets  $\mathcal F\subseteq C$ and $\mathcal P\subseteq \underline{K}(C)$, any $\ep>0$ and $\gamma>0$, any homomorphism $$\Gamma: \Z^k \to U_0(M_n(A))/CU(M_n(A)),$$ there is a unitary $w\in A$ such that
\beq
\|[\phi(f), w]\|<\ep\quad\forall f\in\F,\,\
\mathrm{Bott}(\phi, w)|_{\mathcal P}=0,\andeqn
\eneq
$$\mathrm{dist}(
\overline{\langle({\mathbf 1}_n-\phi(p_i)+\phi(p_i)\tilde{w})({\mathbf 1}_n-\phi(q_i)+\phi(q_i)\tilde{w}^*\rangle},
 \Gamma(x_i))<\gamma, \quad\forall 1\leq i\leq k,$$
{ where $\tilde{w}=\mathrm{diag}(\overbrace{w, ..., w}^n)$.}

\end{thm}

\begin{proof}

Without loss of generality, we may assume that
$\ep<1/2.$
Denote by $$\Delta(a)=\inf\{\mu_{\tau\circ\phi}(O_a);\ \tau\in\mathrm{T}(A), \textrm{$O_a$ is an open ball of $X$ with radius $a$}\}.$$
Since $A$ is simple and $\mathrm{T}(A)$ is compact, $\Delta(a)$ is a nondecreasing function from $(0, 1)$ to $(0, 1)$.

Let $\eta'>0$, $\delta'>0$, $\mathcal G'\subseteq C$, $\mathcal H'\subseteq C_{s.a.}$, $\mathcal P'\subseteq \underline{K}(C)$, $\mathcal U'\subset U_c(K_1(C))$, $\gamma_1>0$, $\gamma_2>0$ be the finite subsets and constants of Theorem 5.3 of \cite{Lin-AU11} with respect to $X$, $\Delta(r/3)/2$, $\mathcal F$, and $\ep/2$. Without loss of generality, one may assume that $\F\subseteq\mathcal G'$ and $\delta'<\ep/4$.

Let $\delta''$ and $\mathcal H''\subseteq C$ be the constant and finite subset of lemma 3.4 of \cite{Lin-AU11} with respect to $X$, $\Delta$, and $\eta'/3$.

Since $X$ is an inverse limit of finite CW-complexes, there is a C*-algebra $C'\cong\mathrm{C}(X')$ for a finite CW-complex $X'$ and a homomorphism $\iota: C' \to C$ such that
\beq
G\subseteq \iota_{*0}(K_0(C')),\,\,\,\quad \{p_i, q_i;\ i=1, ..., k\}\subseteq \iota(M_n(C')),\andeqn\\
{\cal P}'\subset  [\imath]({\cal P}_1'),
\eneq
where ${\cal P}_1'\subset \underline{K}(C')$ is a
finite subset.

Furthermore, one may choose $X'$ such that there is a completely positive linear map $\pi: C\to C'$ so that if $\psi: C'\to A$ is $(\mathcal G'')$-$\delta'/2$-multiplicative  (for some finite subset ${\cal G}''\subset C',$) then $\psi\circ\pi$ is $\mathcal G'$-$\delta'$-multiplicative, and moreover, $$\|\iota\circ\pi(f)-f\|<\min\{\ep/8, \gamma_1/4\},\quad\forall f\in\F\cup \mathcal H'\cup\mathcal H'', $$
and $[\pi]({\mathcal P}')\subseteq {\cal P}_1'$ is well defined.

Denote by $\mathcal P''={\mathcal P'}_1\cup \boldsymbol{\bt}({\mathcal P'}_1)\subseteq\underline{K}(C'\otimes\mathrm{C}(\mathbb T))$, and then denote by $N_1$ the integers of Lemma 9.6 of \cite{Lin-LAH} with respect to $C'\otimes\mathrm{C}(\mathbb T)$, $\pi(\mathcal G')\otimes\{1, z\}$ (in the place of $\mathcal G$), $\delta'/2$ (in the place of $\delta$), and $\mathcal P''$ (in the place of $\mathcal P$), where $z$ denotes the identity function on $\mathbb T$.

Let $M$ (in place of $N$) be the constant of Theorem 2.1 of \cite{Li-interval} with respect to $X'$, $\mathcal H'\cup\mathcal H''$ and $\gamma_1'/2$. Without loss of generality, one may assume that $M>8/(N_1\gamma)$.

Set
\beq
u_i=(({\mathbf 1}_n-p_i)+(p_i\otimes z))(({\mathbf 1}_n-q_i)+(q_i\otimes z))^*
\,\,\, i=1,2,...,k.
\eneq
We may assume that
there are unitaries  $u_i'\in M_n(C')$ such that
${\imath}(u_i')=u_i,$ $i=1,2,...,k.$

{By Theorem 7.6 of \cite{Lin-LAH}, the group $U_0(A)/CU(A)$ is canonically isomorphic to the group $U_0(M_n(A))/CU(M_n(A)).$}  {Therefore we can choose} unitaries $v_i\in U_0(A)$ such that $\Gamma(x_i)=\overline{v_i+{1_{n-1}}}$ for each $1\leq i\leq k$, and choose $T>0$ such that $\mathrm{cel}(v_i)<T,\quad i=1, ..., k.$

Also write $K_1(C')=\Z^{t}\oplus\mathrm{Tor}(K_1(C'))$ and $K_0(C')=\Z^{k'}\oplus\mathrm{Tor}(K_0(C'))$, and let $$\{y_1=[e_1], ..., y_{r'}=[e_{r'}], y_{r'+1}, ..., y_{k'}\}$$ be the standard generators of $\Z^{k'}$ with $y_i\in\ker\rho_{C'},$  $i=r'+1,...,k'$, and $e_i$, $i=1, ..., r'$, projections.

By choosing a larger ${\cal G}''$ and a smaller
$\dt',$ we may assume that, for any unital ${\cal G}''\cup\{1, z\}$-$\dt'$-multiplicative \morp\, $L'$ from $C'$ to an arbitrary C*-algebra induces a well-defined
\hm\, on $s_1(K_1(C'\otimes C(\T))).$

Since $TR(A)\leq 1$, there is an interval algebra $I\subset A$ with $p=1_I$ and $\mathcal G''$-$\delta'/4$-multiplicative completely positive linear maps $L_0: C'\to (1-p)A(1-p)$ and $L_1: C'\to I$ such that
\begin{enumerate}
\item $\|(L_0(\pi(f))+L_1(\pi(f)))-\phi(f)\|<\min\{\ep/8, \dt'/16, \gamma_1/4\}$, for any $f\in \F\cup\mathcal H'\cup\mathcal H''$,
\item $[\phi]|_{\mathcal P'}=[L_0\circ\pi]|_{\mathcal P'}+[L_1\circ\pi]|_{\mathcal P'}$,
\item $I=\bigoplus_i M_{n_i}(\mathrm{C}([0, 1]))$ with $n_i>\max\{16(\mathrm{dim}(X)+1)N_1/\gamma_1, 2M-2N_1(\mathrm{dim}(X)+1)\}$,
\item there are unitaries $v_i'\in (1-p)A(1-p)$ and $v_i'' \in I$ such that $\mathrm{cel}(v'_i\oplus p)<\gamma/4$ in $A$ (by Lemma \ref{lem64}) and $$\|v_i-(v_i'+v_i'')\|<\gamma/4,\quad i=1, ..., k,$$
\item \label{only-pf} moreover, by applying 2.21 of \cite{Lin-LAH}, one may assume that for any $r'+1\leq i\leq k'$
$$|\tau(L_1(y_i))|<\gamma_1/8N_2,\quad \forall\tau\in\mathrm{T}(I).$$
\end{enumerate}
There is a subgroup $G_0\subset \Z^{k'}\subset K_0(C')$ such that
$G_0\cong \Z^k$ and generators $\{g_1, g_2,...,g_k\}\subset G_0$ such that
$\imath_{*0}(g_i)=x_i,$ $i=1,2,...,k.$ Without loss of generality,
we may assume that $[u_i']=g_i,$ $i=1,2,...,k.$
Define a \hm\, $\Gamma_1: K_0(C')\to U_0(I)/CU(I)$ as follows:
    First define $\Gamma_1(g_i)=\overline{v_i''},$ $i=1,2,...,k.$
    This gives a \hm\, from $G_0\to U_0(I)/CU(I).$ Since $U_0(I)/CU(I)$
    is divisible, it extends to a \hm\,  $\Gamma_1$ from $K_0(C')$ to $U_0(I)/CU(I).$
    Note that since $U_0(I)/CU(I)$ is also torsion free,
    $\Gamma_1|_{{\mathrm{Tor}}(K_0(C'))}=0.$

Denote by $m_i=n_i\gamma_1/8+2N_1(\mathrm{dim}(X)+1)$. Note that $$n_i-m_i>M\quad\mathrm{and}\quad m_i/n_i<\gamma_1/4.$$ By Theorem 2.1 of \cite{Li-interval}, there is a homomorphism $$\Psi: C'\to\bigoplus_i M_{n_i-m_i}(\mathrm{C}([0, 1]))\subseteq I$$ such that
$$|\tau\circ\Psi(h)-\tau\circ L_1(\pi(h))|<\gamma_1/2,\quad \forall h\in\mathcal H\cup\mathcal H',\ \forall \tau\in\mathrm{T}(I).$$
Define
$$\kappa=([L_1]-[\Psi])\oplus 0\in\mathrm{Hom}_{\Lambda}(\underline{K}(C'\otimes\mathrm{C}(\mathbb T)), \underline{K}(A)),$$ where $\underline{K}(C'\otimes\mathrm{C}(\mathbb T))$ is identified as $\underline{K}(C')\oplus \boldsymbol{\beta}(\underline{K}(C'))$.

Note that $K_1(C'\otimes\mathrm{C}(\mathbb T))\cong
K_1(C')\oplus K_0(C').$ It may also be written as
$\Z^t\oplus\Z^{k'}\oplus\mathrm{Tor}(K_1(C'\otimes \mathrm{C}(\T))),$
where $k'$ is the rank of of $K_0(C')$.

Define a map $\lambda: \Z^t\oplus \Z^{k'}\to U_0(I)/CU(I)$ as follows:
\beq
\lambda(x)=L^{\ddag}\circ s_1(x)(\Psi^{\ddag}(x^*))\tforal x\in K_1(C')
\andeqn\\
\lambda|_{\Z^{k'}}=\Gamma_1|_{\Z^{k'}}.
\eneq

Note that for any $\tau\in\mathrm{T}(I)$ and any $i=r'+1, ..., k'$, one has that
$$
|\tau(\kappa(y_i))|
=|\tau(L_1(y_i))-\tau(\Psi(y_i))|
=|\tau(L_1(y_i))|<\delta.$$
 By Lemma 9.6 of \cite{Lin-LAH}, there is a ${\mathcal G}''\otimes\{1, z\}$-$\delta'/4$-multiplicative map
$$\Phi: C'\otimes\mathrm{C}(\mathbb T)\to \bigoplus_i M_{m_i}(\mathrm{C}([0, 1]))$$ such that
$$[\Phi]=\kappa\quad \textrm{and}\quad \Phi^{\ddag}|_{s_1(\Z^t\oplus\Z^{k'})}=\lambda.$$

Denote by $$w'=(1-p)\oplus \langle \Phi(1\oplus z)\rangle \oplus (\bigoplus_i 1_{M_{n_i-m_i}})$$ and $\psi: C'\to A$ by
$$\psi= L_0\oplus\Phi|_{C'\otimes 1}\oplus\Psi.$$

Since $\Phi$ is ${\mathcal G}''\otimes\{1, z\}$-$\delta'/4$-multiplicative, it is clear that
$$\|[\psi(\pi(f)), w']\|<\ep/4$$ and
$$\mathrm{Bott}(\psi\circ\pi, w')=\kappa\circ\beta\circ\pi=0.$$
Moreover, in ${U(M_n(A))}/CU(M_n(A))$, one has
\begin{eqnarray*}
 &&\mathrm{dist}(\overline{\langle \psi \otimes \id_n(u_i')\rangle},
\Gamma(x_i)) \\
&\approx_{\gamma/4} &
\mathrm{dist}(\overline{(1-p)\otimes 1_n \oplus \langle\Phi\otimes \id_n(u_i')\rangle
 \oplus
 (\bigoplus_i {1_{n_i-m}})\otimes 1_n},
 \overline{(v_i'\oplus v_i'')\oplus  1_{n-1}}) \\
&=&\mathrm{dist}(\overline{(1-p)\otimes 1_n\oplus\Gamma_1([u_i'])\oplus (\bigoplus_i {1_{n_i-m}})\otimes 1_n}, \overline{v_i'\oplus v_i''\oplus 1_{n-1}})\\
&=&\mathrm{dist}(\overline{(1-p)\otimes 1_n\oplus v_i''\oplus(p\otimes 1_{n-1})}, \overline{v_i'\oplus (1-p)\otimes 1_{n-1}\oplus v_i''\oplus p\otimes 1_{n-1}})\\
&=&\mathrm{dist}(\overline{1_n}, \overline{v_i'\oplus (1_{n-1}+p))}\approx_{\gamma/4} 0.
\end{eqnarray*}

On the other hand, the map $\psi\circ\pi$ is $\mathcal G'$-$\delta'$-multiplicative, and
$$[\psi\circ\pi]|_{\mathcal P'}=[L_0\circ\pi]|_{\mathcal P'}+[\Psi|_{C\otimes 1}\circ\pi]|_{\mathcal P'}+[\Phi\circ\pi]|_{\mathcal P'}=[L_0\circ\pi]|_{\mathcal P'}+[L_1\circ\pi]|_{\mathcal P'}=[\phi]|_{\mathcal P'}.$$
One also has that, for any $u\in\mathcal U'$,
\begin{eqnarray*}
&& \mathrm{dist}(\overline{\phi(u)}, \overline{\langle \psi(\pi(u)) \rangle})\\
&\approx_{\gamma_2}& \mathrm{dist}(\overline{ \langle L_0(\pi(u))\oplus L_1(\pi(u))\rangle }, \overline{\langle L_0(\pi(u))\rangle \oplus \langle (L_1(\pi(u))\Psi(\pi(u^*)))\rangle \oplus\Psi(\pi(u))})\\
&=& 0,
\end{eqnarray*}
and for any $h\in \mathcal H'\cup\mathcal H''$,
\begin{eqnarray*}
|\tau(\phi(h))-\tau(\psi(\pi(h)))|&\approx_{\gamma_1/4}& |\tau(L_0(\pi(h))+L_1(\pi(h)))-\tau(L_0(\pi(h))+\Psi(\pi(h)\otimes1)+
\Phi(\pi(h)))|\\
&\approx_{\gamma_1/2}&\tau(\Psi(\pi(h)\otimes1))\approx_{\gamma_1/4}0.
\end{eqnarray*}
It then follows from Lemma 3.4 of \cite{Lin-AU11} that $$\mu_{\tau\circ\psi\circ\pi}(O_r)>\Delta(r/3)/2$$ for any $r>\eta'$.
By Theorem 5.3 of \cite{Lin-AU11}, there is a unitary $v$ such that
$$\|\phi(f)-v^*\psi(\pi(f)) v\|<\ep/2,\quad\forall f\in\F.$$ Then the unitary $w:=v^*w'v$ satisfies the lemma.
\end{proof}

\begin{cor}\label{BB-exi-mat}
The statement of Theorem \ref{BB-exi} still holds if $\mathrm{C}(X)$ is replaced by $PM_n(\mathrm{C}(Y))P$ for a compact subset  $Y$
of a finite CW-complex and a projection $P$ in $M_n(\mathrm{C}(Y))$.
\end{cor}
\begin{proof}
The corollary clearly holds for $C=M_n(C(X))$ (in this case, $X$ is even not required to be finite dimensional). In what follows we will use this case of the corollary to prove the general case.

Assume that $C=PM_n(C(X))P$. {As in the proof of \ref{hp-mat},} without loss of generality, we may assume that $P(x)\not=0$ for all $x\in  X.$ {Since $X$ is a compact subset of a finite CW-complex,} there is an integer $d$ and a projection $Q\in M_d(PM_n(C(X))P)$ such that
$$QM_d(PM_n(C(X))P)Q\cong M_{r}(C(X))$$ for some integer $r$. Note that
$Q(x)\not=0$ for all $x\in X.$

Without loss of generality, one may assume that $P \preceq Q$, that is, there is also a partial isometry $V\in M_d(PM_n(C(X))P)$ such that $VV^*\leq Q$ and $V^*V=\mathrm\{P, 0, ..., 0\}$. In particular, $V$ induces an isomorphism between $PM_n(C(X))P$ and the unital hereditary subalgebra of $QM_d(PM_n(C(X))P)Q$ generated by $VV^*$.

Fix $d$, $Q$, and $V$.

Since $X$ is compact, there is an integer $l\geq 1$ such that $X=X_1\sqcup\cdots\sqcup X_l$ { and} the ranks of the restrictions of $P$ and $Q$ to each $X_i$, $1\leq j\leq l$, are constant. Denote by $P_j$ and $Q_j$ the restrictions of $P$ and $Q$ to $X_j$ respectively.
Let $R=\max_{1\le j\le 1}\{{\rm rank} Q_j\}.$

Let $G\subseteq K_0(C)$ be a finitely generated group with a fixed decomposition $G=\mathbb Z^k\oplus\mathrm{Tor}(G)$ with
 $\Z^k$ generated by
$$
\{x_{1}=[p_1]-[q_1], x_2=[p_2]-[q_2], ..., x_{k}=[p_{k}]-[q_{k}]\},
$$
where $p_i, q_i\in M_m(C)$ (for some integer $m\ge 1$) are projections,
$i=1,...,k$. 

Let $A$ be a unital simple C*-algebra with $TR(A)\leq 1$, and let $\phi: C\to A$ be a monomorphism. Let $\mathcal F\subseteq C$, $\mathcal P\subseteq\underline{K}(C)$, $\ep>0$, $\gamma>0$, and $\Gamma: \mathbb{Z}^k\to U_0(M_m(A))/CU(M_m(A))$ be a homomorphism.

Denote by $q=(\phi\otimes 1_{M_d})(Q)$, $e=\phi\otimes 1_{M_d}(VV^*)\in qM_d(A)q$ and $v=\phi\otimes 1_{M_d}(V)\in qM_d(A)q$.

Choose unitaries $v_1, ..., v_k\in U_0(M_m(A))$ such that $$\Gamma(x_i)=\overline{v_i}\in U_0(M_m(A))/CU(M_m(A)),\quad i=1, ..., k.$$ Then the elements $(v\otimes 1_m)v_i(v\otimes 1_m)^*$, $i=1, ..., k$, are unitaries in $M_m(eM_d(A)e)=(e\otimes 1_m)(M_m(qM_d(A)q))(e\otimes 1_m)$.

Choose $\ep_1>0$ and a finite subset $\mathcal G_1\subseteq QM_d(C)Q$  such that $VV^*\in \mathcal G_1$ and if there is a unitary $u\in qM_d(A)q$ with
$$\|[(\phi\otimes 1_{M_d})|_{QM_d(C)Q}(c), u]\|<\ep_1,\quad\forall c\in\mathcal G_1,$$ then
$$\|[(\phi\otimes 1_{M_d})|_{VCV^*}(VcV^*), u]\|<\ep,\quad\forall c\in\mathcal G.$$
By choosing $\ep_1$ sufficiently small (note that $VV^*\in\mathcal G_1$), the element $v^*uv$ can be assumed to be invertible in $A$ and
\begin{equation}\label{eq-c-1}
\|[\phi(c), \langle v^*uv \rangle]\|<\ep,\quad\forall c\in\mathcal G.
\end{equation}

Using the same argument as that of Corollary \ref{hp-mat}, one may choose a finite subset $\mathcal G_2\subseteq QM_d(C)Q$ and $\ep_2>0$ such that if
$$\|[(\phi\otimes 1_{M_d})|_{QM_d(C)Q}(c), u]\|<\ep_2,\quad\forall c\in\mathcal G_2,$$
and
$$\mathrm{Bott}(\phi\otimes 1_{M_d}|_{QM_d(C)Q}, u)|_{V\mathcal PV^*}=0,$$
then
$$\mathrm{Bott}(\phi\otimes 1_{M_d}|_{VCV^*}, \langle{eue}\rangle)|_{V\mathcal PV^*}=0.$$
Then one may assume further that $\ep_2$ is sufficiently small so that
$\|v^*\langle eue \rangle v-\langle v^*u v \rangle\|$
is small enough so that
\begin{equation}\label{eq-c-2}
\mathrm{Bott}(\phi, \langle{v^*uv}\rangle)|_{\mathcal P}=0.
\end{equation}

Denote by ${\tilde{V}}=V\otimes 1_{m}$ and ${\tilde{v}}=v\otimes 1_{m}$. Note that ${\tilde{V}} p_i {\tilde{V}}^*, {\tilde{V}}q_i{\tilde{V}}^* \in M_m(QM_d(C)Q)$.

Define $$\Gamma': \mathbb{Z}^k\to U_0(M_m(qM_d(A)q))/CU(M_m(qM_d(A)q))$$ by
$$\Gamma'(x_i) = \overline{{\tilde{v}} v_i {\tilde{v}}^* + (q\otimes 1_m-e\otimes 1_m)},\quad i=1, ..., k.$$

One may choose a finite subset $\mathcal G_3\subseteq QM_d(C)Q$ and $\ep_3>0$ such that if there is a unitary $u\in qM_d(A)q$ such that
$$\| [(\phi\otimes 1_{d})|_{QM_d(C)Q}(c), u]\|<\ep_3, \quad\forall c\in\mathcal G_3,$$
and if
\begin{eqnarray*}
&&\mathrm{dist}(
\overline{\langle(q\otimes 1_m-\phi({\tilde{V}}p_i{\tilde{V}}^*)+\phi({\tilde{V}}p_i{\tilde{V}}^*)u\otimes 1_m)(q\otimes 1_m-\phi({\tilde{V}}q_i{\tilde{V}}^*)+\phi({\tilde{V}}q_i{\tilde{V}}^*)u^*\otimes 1_m\rangle},
 \Gamma'(x_i))\\
 &<&\frac{\gamma}{(R+\frac{1}{8})}
 \end{eqnarray*}
 for any $1\leq i\leq k,$
then
\begin{equation}\label{gcu-01}
\mathrm{dist}(g_i, CU(M_m(qM_d(A)q))) <\frac{\gamma}{2(R+\frac{1}{8})},\quad i=1, ..., k,
 \end{equation}
where
$$g_i= \langle(q\otimes 1_m-\phi({\tilde{V}}p_i{\tilde{V}}^*)+\phi({\tilde{V}}p_i{\tilde{V}}^*)u\otimes 1_m)(q\otimes 1_m-\phi({\tilde{V}}q_i{\tilde{V}}^*)+\phi({\tilde{V}}q_i{\tilde{V}}^*)u^*\otimes 1_m\rangle  ({\tilde{v}}v^*_i{\tilde{v}}^*+(q\otimes 1_m-e\otimes 1_m)).$$
One may assume that $\ep_3$ is sufficiently small so that
\begin{equation}\label{gcu-02}
\|g_i - (g_i'+(q\otimes 1_m - e\otimes 1_m) )\| < \frac{\gamma}{2(R+\frac{1}{8})},\quad i=1, ..., k,
\end{equation}
where
$$g_i'=\langle(e\otimes 1_m-\phi({\tilde{V}}p_i{\tilde{V}}^*)+\phi({\tilde{V}}p_i{\tilde{V}}^*)\langle{eue}\rangle\otimes 1_m)(e\otimes 1_m-\phi({\tilde{V}}q_i{\tilde{V}}^*)+\phi({\tilde{V}}q_i{\tilde{V}}^*)\langle{eue}\rangle^*\otimes 1_m\rangle({\tilde{v}}v^*_i{\tilde{v}}^*).$$

{As in the proof of \ref{hp-mat}, since $A$ is a unital simple \CA\, with $TR(A)\le 1,$ one has that $g'\in U_0(M_m(eM_d(A)e)).$}
Note that for any $\tau\in T(M_m(qM_d(A)q))$, one has
$$\tau(e\otimes 1_m)\geq \frac{1}{R},$$
and therefore $R[e\otimes 1_m]\geq [q\otimes1_m-e\otimes 1_m]$. By Lemma 3.3 of \cite{Lin-hmtp}, one has that
$$\mathrm{dist}(g_i', CU(M_m(eM_d(A)e)))<(R+\frac{1}{8})\frac{\gamma}{(R+\frac{1}{8})}=\gamma,\quad i=1, ..., k.$$
Then one may also assume further that $\ep_3$ is sufficiently small so that
\begin{equation}\label{eq-c-3-0}
\mathrm{dist}(
\langle(1_m-\phi(p_i)+\phi(p_i)\langle{v^*uv}\rangle\otimes 1_m)(1_m-\phi(q_i)+\phi(q_i)\langle{v^*uv}\rangle^*\otimes 1_m\rangle v^*_i,
 CU({M_m(A)})<\gamma,
 \end{equation}
for any $1\leq i\leq k.$ That is,
\begin{equation}\label{eq-c-3}
\mathrm{dist}(
\overline{\langle(1_m-\phi(p_i)+\phi(p_i)\langle{v^*uv}\rangle\otimes 1_m)(1_m-\phi(q_i)+\phi(q_i)\langle{v^*uv}\rangle^*\otimes 1_m\rangle},
 \Gamma(x_i))<\gamma,
 \end{equation}
for any $1\leq i\leq k,$

Now, since $Q(M_d(PM_n(C(X))P))Q\cong M_r(C(X))$, applying the corollary to $M_r(C(X))$,  one obtains a unitary $u\in qM_d(A)q$ such that
$$\|[(\phi\otimes 1_{M_d})|_{QM_d(C)Q}(c), u]\|<\min\{\ep_1, \ep_2, \ep_3\},\quad\forall c\in \mathcal G_1\cup\mathcal G_2\cup\mathcal G_3,$$
$$\mathrm{Bott}((\phi\otimes 1_{M_d})|_{QM_d(C)Q}, u)|_{V\mathcal PV^*}=0,$$
and
\begin{eqnarray*}
&&\mathrm{dist}(
\overline{\langle(q\otimes 1_m-\phi({\tilde{V}}p_i{\tilde{V}}^*)+\phi({\tilde{V}}p_i{\tilde{V}}^*)u\otimes 1_m)(q\otimes 1_m-\phi({\tilde{V}}q_i{\tilde{V}}^*)+\phi({\tilde{V}}q_i{\tilde{V}}^*)u^*\otimes 1_m\rangle},
 \Gamma'(x_i))\\
 &<&\frac{\gamma}{(R+\frac{1}{8})}
 \end{eqnarray*}
 for any $1\leq i\leq k,$

By \eqref{eq-c-1}, \eqref{eq-c-2}, and \eqref{eq-c-3}, the unitary $w=\langle v^*uv \rangle \in A$ satisfies the corollary.
\end{proof}

\begin{lem}\label{fix-alg-k1}
Let $C=\mathrm{C}(X)$ with $X$ a compact metric space, and let $A$ be a simple C*-algebra with $TR(A)\leq 1$.  Then, for any $\ep>0$, any finite subset $\F\subseteq C$, any finite subset $\mathcal P\subseteq\underline{K}(C)$, and any finite subset $\mathcal U\subseteq
{U_c(C)\cap U(M_m(C))}$ (for some $m\ge 1$), there are finite subset $\mathcal G\subseteq C$ and $\delta>0$ such that if $h: C\to A$ is a unital homomorphism, and  $\phi: C\to pAp$ is a  unital $\mathcal G$-$\delta$-multiplicative map satisfying
$$
[h(u)]= {[\phi(u)]} \,\,\,{\rm in}\,\,\, K_1(A) \quad\forall u\in \mathcal U.
$$
%
there exists a $\F$-$\ep$-multiplicative map $\Phi: C\to (1-p)A(1-p)$ such that
$$
[\Phi]|_{\mathcal P}=[H]|_{\mathcal P},
$$
where $H: C\to (1-p)A(1-p)$ is the direct sum of finitely many point-evaluations and
$${\mathrm{dist}((h\otimes {\rm id}_{M_m})^\ddag(\overline{u})^{-1}\overline{\langle ((\Phi\otimes \oplus\phi)\otimes {\rm id}_{M_m})(u)\rangle }, \overline{1_m})<\ep,\quad\forall u\in \mathcal U}.$$
\end{lem}
\begin{proof}
Since $C$ can be written as an inductive limit of the C*-algebras of continuous functions on finite CW-complexes, without loss of generality, one may assume that $X$ is a finite CW-complex.

{Therefore we may also assume that $K_*(C)$ is  finitely generated.  Without loss of generality, we may further  assume
that ${\cal U}\in U(M_m(C))$ for some $m\ge 1$ which generates $K_1(C).$ }
 Fix a decomposition $K_1(C)=\Z^r\oplus\mathrm{Tor}(K_1(X))$.

 {{\it In what follows, to simplify notation, for maps  such as $h, $ $\Phi$ and $\phi,$ we will  use $h,\Phi, $ and $\phi$
 in stead of $h\otimes {\rm id}_{M_m},$ $\Phi\otimes {\rm id}_{M_m}$ and $\phi\otimes {\rm id}_{M_m}.$}}

 Choose a subset $\mathcal G\subseteq C$ and $\delta>0$ such that if $\psi': C\to D$ is a $\mathcal G$-$\delta$-multiplicative map for a C*-algebra $D$, then element $\overline{\langle\phi'(u^*)\rangle}$ is well defined for any $u\in\mathcal U$.

Denote by $N_1$ and $N_2$ the constants of Lemma 9.6 of \cite{Lin-LAH} with respect to $X$, $\F$ (in the place of $\mathcal G$), $\ep$ (in the place of $\delta$).


Let $h: C\to A$ be a homomorphism and let $\phi: C\to A$ be a $\mathcal G$-$\delta$-multiplicative map satisfying the lemma.

Since
${[h(u)]=[\phi(u)]}$ in $K_1(A)$ for all  $ u\in \mathcal U$ and $A$ has stable rank one, there is $T>0$ such that
$$\mathrm{cel}{(h(u)(\langle (1-p)\otimes 1_m\oplus\phi(u^*)\rangle ))}<T,\quad\forall u\in\mathcal U.$$

Since $TR(A)\leq 1$, there is a interval algebra $I\subseteq A$ with $q=1_I$ and $\mathcal G$-$\delta$-multiplicative maps $h_0, \phi_0: C\to (1-q)A(1-q)$, $h_1, \phi_1: C\to I$ such that
\begin{enumerate}
\item $\|h(u)-(h_0(u)\oplus h_1(u))\|<\ep/4$ and $\|\phi(u)-(\phi_0(u)\oplus \phi_1(u))\|<\ep/4$ for any $u\in\mathcal{U}$.
\item The element $p_0:=h_0(1_C)-\phi_0(1_C)$ is a projection in $(1-q)A(1-q)$, and the element  $p_1:=h_1(1_C)-\phi_1(1_C)$ is a projection in $I$; moreover,  $p_0+p_1=1-p$.

\item $\mathrm{dist}(\overline{\langle h_0(u)\oplus q\otimes 1_m\rangle \langle (\phi_0(u^*)\oplus (p_0\oplus q)\otimes 1_m) \rangle }, \overline{1_m})<\ep/4$,
that is,
$$\mathrm{dist}(\overline{\langle h_0(u)\oplus q\otimes 1_m\rangle}, \overline{\langle (\phi_0(u)\oplus (p_0\oplus q)\otimes 1_m) \rangle })<\ep/4.$$

\item The rank of $p_1$ is at least $N_1(\mathrm{dim}(X)+1)$.
\end{enumerate}

Then it follows from Lemma 9.6 of \cite{Lin-LAH} that there is a $\F$-$\ep$-multiplicative map $\Phi': C\to p_1Ip_1$ such that
$$
[\Phi']|_{\mathcal P}=[H']|_{\mathcal P},
$$
where $H': C\to p_1Ip_1$ is the direct sum of finitely many point evaluation maps and $$(\Phi')^\ddag(u)=h_1^\ddag(u) \overline{\langle \phi_1(u^*) \rangle},\quad\forall u\in\mathcal U.$$

Let $\Phi_0: C\to (1-q)A(1-q)$ be the map $f\to f(\xi)(1-p_0)$ for some $\xi\in X$, and define
$$\Phi=\Phi_0\oplus\Phi':\ C\to (p_0+p_1)A(p_0+p_1)=(1-p)A(1-p).$$ It is clear that $$[\Phi]|_{\mathcal P}=[H]|_{\mathcal P},$$ where $H: C\to (1-p)A(1-p)$ is a direct sum of finitely many point-evaluations. Furthermore, for any $u\in\mathcal U$, one has
\begin{eqnarray*}
\mathrm{dist}(\overline{\langle \Phi(u) \rangle \oplus\phi(u)}, \overline{h(u)})&<&\mathrm{dist}(\overline{\langle ((\Phi_0(u)\oplus\phi_0(u))\oplus (\Phi'(u)\oplus\phi_1(u))\rangle}, \overline{\langle h_0(u)\oplus h_1(u) \rangle})+\ep/4\\
&\leq&\mathrm{dist}(\overline{\langle \phi_0(u)\oplus (p_0\oplus  q)\otimes 1_m \rangle}, \overline{\langle h_0(u)\oplus q\otimes 1_m \rangle})+\\
&&\mathrm{dist}(\overline{\langle (1-q)\otimes 1_m\oplus\Phi'(u)\oplus\phi_1(u) \rangle}, \overline{\langle (1-q)\otimes 1_m\oplus h_1(u) \rangle})+\ep/4\\
&\leq&\mathrm{dist}(\overline{\langle \Phi'(u)\oplus\phi_1(u) \rangle}, \overline{ \langle h_1(u)\rangle})+3\ep/4 \\
&=&3\ep/4<\ep.
\end{eqnarray*}
This proves the lemma.
\end{proof}

\begin{thm}\label{B1B2-alg}
Let $C$ be an AH-algebra, and let $A$ be a simple C*-algebra with $TR(A)\leq 1$.  Suppose that $h: C\to A$ is a monomorphism. Then, for any $\ep>0$, any finite subset $\F\subseteq C$ and any finite subset $\mathcal{P}\subseteq\underline{K}(C)$, there is a C*-algebra $C'\cong PM_n(\mathrm{C}(X'))P$ for some finite CW-complex $X'$ with
$K_1(C')=\Z^k\oplus {\mathrm{Tor}}(K_1(C'))$ and a homomorphism $\iota: C'\to C$ with  $\mathcal P\subseteq[\iota](\underline{K}(C'))$, a finite subset $\mathcal{Q}\subseteq \Z^k \subset K_1(C')$ and $\delta>0$ satisfying the following: Suppose that $\kappa\in\mathrm{Hom}_\Lambda(\underline{K}(C'\otimes {\mathrm{C}}(\T)), \underline{K}(A))$ with
$$|\rho_A\circ\kappa(\boldsymbol{\beta}(x))(\tau)|<\delta,\quad\forall x\in\mathcal{Q},\ \forall \tau\in\mathrm{T}(A).$$
Then there exists a unitary $u\in A$ such that
$$\|[h(c), u]\|<\ep,\quad\forall c\in\F\quad\textrm{and}\quad \mathrm{Bott}(h\circ\iota, u)=\kappa\circ\boldsymbol{\beta}.$$

Moreover, there is a sequence of C*-algebras $C_n$ with the form $C_n=P_nM_{r(n)}(C(X_n))P_n$, where each $X_n$ is a finite CW-complex and $P_n\in M_{r(n)}(C(X_n))$ a projection, such that $C=\varinjlim(C_n, \phi_n)$ for a sequence of unital homomorphisms $\phi_n: C_n\to C_{n+1}$ and one may choose $C'=C_n$ and $\iota=\phi_n$ for some integer $n\geq 1$.
\end{thm}


\begin{proof}
The proof is similar to that of Theorem 6.3 of \cite{Lnclasn}. Without loss of generality, one may assume that $C=\mathrm{C}(X)$ for some compact metric space $X$.  Denote by
$$\Delta(a)=\inf\{\mu_{\tau\circ h}(O_a);\ \tau\in\mathrm{T}(A), \textrm{$O_a$ is an open ball of $X$ with radius $a$}\}.
$$ Since $A$ is simple and $\mathrm{T}(A)$ is compact, $\Delta(a)$ is a nondecreasing function from $(0, 1)$ to $(0, 1)$.

Let $B$ be a unital separable simple amenable C*-algebra with $TR(B)=0$ satisfying the UCT and
$$(K_0(B), K_0^+(B), [1_B], K_1(B)) \cong (K_0(A), K_0^+(A), [1_A], K_1(A)).$$

Then there is an embedding $\iota': B\to A$ such that $[\iota']$ induces an identification of the above. In the following, we identify $B$ as a C*-subalgebra of $A$.

Let $\ep_1>0$ with $\ep_1<\ep$, and let $\F_1\supseteq\F$ be a finite subset such that for any unital homomorphism $H: C\to A$ and unitary $u'\in A$ satisfying
$$\|[H(c), u']\|<\ep_1,\quad\forall c\in\F_1,$$
the map $\mathrm{Bott}(H, u')|_{\mathcal P}$ is well defined; moreover, if $H': C\to A$ is any other unital monomorphism satisfying $$\|H(c)- H'(c)\|<\ep_1,\quad\forall c\in\F_1,$$
then
$$\mathrm{Bott}(H, u')|_{\mathcal P}=\mathrm{Bott}(H', u')|_{\mathcal P}.$$

Let $\eta$, $\delta_1$ (in the place of $\delta$), $\mathcal G_1\subseteq C$ (in the place of $\mathcal H$), $\mathcal P'\subseteq\underline{K}(C)$ (in the place of $\mathcal P$), and $\mathcal U\subseteq U_c(C)$, $\gamma_1$, $\gamma_2$ be the constants and finite subsets of Theorem 5.3 of \cite{Lin-AU11} with respect to $\ep_1/2$, $\F_1$, and $\Delta(\cdot/3)/2$.

Let $\delta_2$ and $\mathcal G_2'\subseteq C$ be the constant and finite subset required by Lemma 3.4 of \cite{Lin-AU11} with respect to $\Delta$, $\eta$, and $\lambda_1=\lambda_2=1/2$. Denote by $\mathcal G_2=\mathcal G_1\cup\mathcal G_2'$. Without loss of generality, one may assume that $\delta_2<\gamma_1$.

{Let ${\cal G}_3$ (in place of ${\cal G}$) be a finite subset of $C$ and $\dt_3>0$ be as required by Lemma
\ref{fix-alg-k1} for ${\cal G}_1$
(in place of ${\cal F}$), $\dt_1$ (in place of $\ep$),  $\mathcal P'$ (in place of $\mathcal P$), and $\mathcal U$ . Without loss of generality, we may assume that ${\cal G}_3\subset {\cal G}_1$
and $\dt_3<\min\{\dt_1, \dt_2\}.$ }

By Lemma 6.2 of \cite{Lnclasn}, there is a $\mathcal G_3$-$\delta_3$-multiplicative map $h_0: C\to p_0Bp_0$ with $\tau(p_0)<\delta_2/4$ and a unital homomorphism $h_1': C\to F$, where $F$ is a finite dimensional C*-subalgebra of $B$ with $1_F=1-p_0$ such that
$$[h_0+h_1']|_{\mathcal P'}=[h]|_{\mathcal P'}\ \mathrm{in}\ KL(C, A).$$

Let $C'\subseteq  C$,  $1>\delta_4>0$ and $\mathcal Q'
\subseteq K_1(C')$ (in place of ${\cal Q}$)
be the constant and finite subset required by
Lemma 6.1 of \cite{Lnclasn} with respect to $\F$, $\mathcal P$, and $p_0Bp_0$.
We may write
$K_1(C')=\Z^k\oplus {\mathrm{Tor}}(K_1(C')).$
Let ${\cal Q}\subset \Z^k$ be a finite subset such that
$$
{\cal Q}'=\{x+y: x\in {\cal Q}\andeqn  y\in {\mathrm{Tor}}(K_1(C'))\}.
$$
Let $\delta=\min\{\delta_4\delta_1/16\pi, \delta_4\delta_2/4\}$. Now let $\kappa\in\mathrm{Hom}_\Lambda(\underline{K}(C'\otimes\mathrm{C}(\mathbb T)), \underline{K}(A))$ with
$$|\rho_A\circ\kappa(\boldsymbol{\beta}(x))(\tau)|<\delta \rforal  x\in\mathcal{Q}\andeqn  \rforal \tau\in\mathrm{T}(A).$$
Note that this implies
$$
|\rho_A\circ \kappa(\boldsymbol{\bt}(x))(\tau)|<\dt\tforal x\in {\cal Q}'
$$
and for all $\tau\in T(A).$
By Lemma 6.1 of \cite{Lnclasn}, there is a unitary $u_0\in p_0Bp_0$ such that
$$\|[h_0(c), u_0]\|<\ep_1/2,\quad \forall c\in \F,$$ and
$$\mathrm{Bott}(h_0\circ\iota, u_0)=\kappa\circ\boldsymbol{\beta}.$$

Put $u=u_0+(1-p_0)$. Then there is a nonzero projection $q_0\in(1-p_0)A(1-p_0)$ such that
$$q_0f=fq_0,\quad\forall f\in F,\quad\textrm{and}\quad \tau(q_0)<\delta,\quad\forall \tau\in T(A).$$

Define $\psi_0(c)=q_0h_1'(c)$ and $\psi'_0(c)=(1-p_0-q_0)h_1'$ for all $c\in C$. By Lemma 9.5 of \cite{LinTAI}, there is C*-subalgebra $B_0\in (1-p_0-q_0)A(1-p_0-q_0)$ with $B_0$ an interval algebra and a unital homomorphism $h_1: C\to B_0$ such that
$(h_1)_{*0}=(\psi_0')_{*0}$ and
$$|\tau(h_1(f))-\tau(1-p_0-q_0)\tau(h(f))|<\delta_2/4,\quad\forall f\in\mathcal G_2.$$

Define $\psi_1=h_0+h_1$. By Lemma \ref{fix-alg-k1}, there is $\mathcal G_1$-$\delta_1$-multiplicative map $\Phi: C\to q_0Aq_0$ with $\Phi_{*0}=(\psi_0)_{*0}$, $[\Phi]|_{\mathcal P'}=[H]|_{\mathcal P'}$ in $KL(C, A)$ for some point evaluation map $H: C\to pAp$, and
$$\mathrm{dist}(h^\ddag(\overline{u})^{-1}\overline{\langle (\Phi\oplus\psi_1)(u) \rangle}, \overline{{1_m}})<\gamma_2,\quad\forall u\in \mathcal U.$$

Define $h_2=\Phi\oplus\psi_1$. Then
$[h_2]|_{\mathcal P'}=[h]|_{\mathcal P'}$ in $KL(C, A)$
and for any $u\in \mathcal U$,
$$\mathrm{dist}(\overline{\langle h_2(u) \rangle}, \overline{h(u)})=\mathrm{dist}(\overline{\langle \Phi(u)\oplus\psi_1(u) \rangle}, \overline{h(u)})\approx_{\gamma_2} 0.$$

Moreover, for any $f\in \mathcal G_2$ and any $\tau\in T(A)$,
\begin{eqnarray*}
|\tau(h(f))-\tau(h_2(f))| & < & \delta_2/4 +|\tau(h(f))-\tau(h_1(f))|\\
&\leq & 3\delta_2/4 + |\tau(1-p_0-q_0)\tau(h(f))-\tau(h_1(f))|\\
&<&\delta_2 <\gamma_1.
\end{eqnarray*}

Note that $\mu_{\tau\circ h}(O_a)\geq \Delta(a)$ for any $a$, by Lemma 3.4 of \cite{Lin-AU11}, one has
$$\mu_{\tau\circ h_2}(O_a)\geq\frac{1}{2}\Delta(a/3)$$
for any $a\geq \eta$. Then, by Theorem 5.3 of \cite{Lin-AU11}, there is a unitary $U\in A$ such that
$$\mathrm{ad}U\circ h_2\approx_{\ep_1/2} h,\quad\textrm{on $\F_1$}.$$
Define $u= U^*(u_0+(1-p_0))U.$ Then
$$\|[h(c), u]\|<\ep_1,\quad\forall c\in\F_1.$$
Moreover, by the choice of $\ep_1$, one has
$$\mathrm{Bott}(h\circ\iota, u)=\mathrm{Bott}(h_2\circ\iota, u_0+(1-p_0))=\mathrm{Bott}(h_0\circ\iota, u_0)=\kappa\circ\boldsymbol{\beta},$$
as desired.
\end{proof}

\section{Asymptotic unitary equivalence}

\begin{lem}\label{asy-01-tr0}
Let $C$ be a unital AH-algebra and let $A$ be a unital separable simple C*-algebra with $TR(A) \leq 1$. Suppose that $\phi_1, \phi_2: C\to A$ are two unital monomorphisms. Suppose that
\begin{enumerate}
\item $[\phi_1]=[\phi_2]$ in $KL(C, A)$, $\phi_1^\ddag=\phi_2^\ddag$, $(\phi_1)_\sharp=(\phi_2)_\sharp$,
\item $R_{\phi_1, \phi_2}(K_1(M_{\phi_1, \phi_2}))\subseteq\rho_A(K_0(A))$.
\end{enumerate}
Then, for any increasing sequence of finite subsets $(\mathcal F_n)$ of $C$ whose union is dense in $C$, any increasing sequence of finite subsets $(\mathcal P_n)$ of $K_1(C)$ with $\bigcup_{n=1}^\infty\mathcal P_n=K_1(C)$ and any decreasing sequence of positive number $(\delta_n)$ with $\sum_{n=1}^\infty\delta_n<\infty$, there exists a sequence of unitaries $(u_n)$ in $U(A)$ such that
$$\mathrm{ad}(u_n)\circ\phi_1\approx_{\delta_n} \phi_2\quad\textrm{on $\F_n$},$$
and
$$\rho_A(\mathrm{bott}_1(\phi_2, u_n^*u_{n+1})(x))=0,$$ for all $x\in\mathcal P_n$ for all sufficiently large $n$.
\end{lem}

\begin{proof}
The proof is a simple modification of the proof of Lemma 7.1 of \cite{Lnclasn}. In the place of Theorem 6.3 of \cite{Lnclasn} being used, one uses the second part of Theorem \ref{B1B2-alg} instead.
\end{proof}

\begin{thm}\label{MT1}
Let $C$ be a unital AH-algebra and let A be a unital separable simple C*-algebra with $TR(A)\leq1$. Suppose that $\phi_1, \phi_2: C\to A$ are two unital monomorphisms. Then there exists a continuous path of unitaries $\{u(t) : t\in [0, \infty)\} \subseteq A$ such that
$$\lim_{t\to\infty}\mathrm{ad}(U(t))\circ\phi_1(c)=\phi_2(c)\quad \textrm{for all $c\in C$} $$
if and only if
$$[\phi_1]=[\phi_2]\ \textrm{in}\ KK(C, A),\ (\phi_1)^{\ddag}=(\phi_2)^\ddag,\ (\phi_1)_\sharp=(\phi_2)_\sharp$$ and
$$\overline{R}_{\phi_1, \phi_2}=0.$$
\end{thm}

\begin{proof}
We only have to show the ``if" part.

Let $C=\varinjlim(C_n, \psi_n)$, where $C_n$ is a C*-algebra in the form of $P_nM_{r(n)}(C(X_n))P_n$ with $X_n$ having a finite covering dimension, and $\phi_n: C_n\to C_{n+1}$ is a unital monomorphism. Let $(\F_n)$ be an increasing sequence of finite subsets of $C$ such that $\bigcup _{n=1}^\infty \F_n$ is dense in $C$.

For each $n$ and $0<a<1$, define $$\Delta_n(a)=\inf\{\mu_{\tau\circ\phi_1}(O_a):\ \textrm{$O_a$ an open ball of $X_n$ with radius $a$}\}.$$
Since $A$ is simple, one has that $\Delta_n(a)\in(0, 1)$ for any $a\in (0, 1)$.

Consider the mapping torus $$M_{\phi_1, \phi_2}=\{f\in C([0,1], A): \textrm{$f(0)=\phi_1(a)$ and $f(1)=\phi_2(a)$ for some $a\in C$}\}.$$

Since $C$ satisfies the Universal Coefficient Theorem, the assumption of $[\phi_1]=[\phi_2]$ in $KK(C, A)$ implies the following short exact sequence splits:
$$
0\to  \underline{K}(SA) \to   \underline{K}(M_{\phi_1, \phi_2}) \to^{\pi_0} \underline{K}(C) \to 0.
$$
Denote by $\theta: \underline{K}(C)\to\underline{K}(M_{\phi_1, \phi_2})$ the splitting map.

Since $\tau\circ\phi_1=\tau\circ\phi_2$ for all $\tau\in T(A)$ and $\overline{R}_{\phi_1, \phi_2}=0$, we may also assume that
$$R_{\phi_1, \phi_2}(\theta(x))=0,$$ for all $x\in K_1(C)$.

In what follows, for any C*-algebras $C''$ and $A$ and a homomorphism $\phi: C''\to A$, for any $x=[p]-[q]\in K_0(C)$ with projections $p, q\in M_n(A)$  (for some integer $n\geq 1$) and a unitary $u\in A$ with $\|[\phi(p), { \tilde{u}}]\|<1/4$ and $\|[\phi(q), { \tilde{u}}]\|<1/4$, where ${ \tilde{u}}=\mathrm{diag}(u, ..., u)$, define
\begin{equation}\label{bu-notation}
g_{x, u}^{\phi}:=\overline{\langle (1_{n}-\phi(p)+ \phi(p){ \tilde{u}})(1_{n}-\phi(q)+ \phi(q){ \tilde{u}}^*) \rangle}\in U_n(A)/CU_n(A).
\end{equation}

Let $\delta'_n>0$(in place of $\delta$), $\eta'_n$ (in place of $\eta$), $\gamma_n'$ (in place of $\gamma$), $\mathcal G_n'\subseteq C_n$ (in place of $\mathcal G$), $\mathcal P'_n\subseteq\underline{K}(C_n)$ and $\mathcal Q'_n=\{x_{n,1}, ..., x_{n, m(n)}\}\subseteq K_0(C_n)$ (in place of $\mathcal Q$) be the constants and  finite subsets corresponding to $1/2^{n+1}$, $\mathcal F_n$ and $\Delta_n$ required by Theorem \ref{hmtp-H}. Without loss of generality, one may assume that $\mathcal [\psi_{n, n+1}](P'_n)\subseteq\mathcal P'_{n+1}$ for all $n$. Note that $\{x_{n,1}, ..., x_{n, m(n)}\}$ are free (hence generate a group $\Z^{m(n)}\subseteq K_0(C_n)$), and write $x_{n, j}=[p_{n, j}]-[q_{n, j}]$ for some projections $p_{n, j}, q_{n, j}\in M_{l(n)}(C_n)$.

Consider the image $[\psi_{n, n+1}](\Z^{m(n)})$, and fix a decomposition $$[\psi_{n, n+1}](\Z^{m(n)})=\Z^{k(n)}\oplus\mathrm{Tor}([\psi_{n, n+1}](\Z^{m(n)}))$$ for some integer $k(n)$. One also fixes a lifting of $\Z^{k(n)}$ in $\Z^{m(n)}$. Write $\{y_{n, 1}, y_{n, 2}, ..., y_{n, k(n)}\}$ a set of generators of $\Z^{k(n)}$, and $\{y_{n, 1}', y_{n, 2}',..., y_{n, k(n)}'\}$ the corresponding elements in $\Z^{m(n)}$. Note that there are integers $c_{i, j}^{(n)}$ such that $$ x_{n, i}=\sum_{j=1}^{k(n)} c_{i, j}^{(n)} y'_{n, j} + r_{i},\quad i=1, ..., m(n)$$ with $[\psi_{n, n+1}](r_i)$ a torsion element in $K_0(C_{n+1})$.

Since $TR(M_{l(n)})\leq 1$, the group $U_0(M_{l(n)}(A))/CU(M_{l(n)}(A))$ is torsion free (Theorem 6.11 of \cite{LinTAI}). Therefore, without loss of generality, one may assume that $\delta'_n$ is sufficiently small and $\mathcal G'_n$ is sufficiently large such that if $h': C\to A$ is a homomorphism and $u'\in A$ a unitary with $\|[h'(a), u']\|<\delta'_n$ for all $a\in\mathcal G_n'$,
and if $$g_{x_{n, i}, u'}^{h'},\ g_{y_{n, j}', u'}^{h'} \in U_0(M_{l(n)}(A))/CU(M_{l(n)}(A)),\quad 1\leq i\leq m(n), 1\leq j\leq k(n),$$
then
\begin{equation}\label{ak-tor}
\mathrm{dist}(g_{x_{n, i}, u'}^{h'},  \prod_{j=1}^{k(n)} (g_{y_{n, j}', u'}^{h'})^{c_{i, j}^{(n)}} )<\gamma'_n/8, \quad i=1, ..., m(n).
\end{equation}

We also assumes that $\mathrm{Bott}(h', u')|_{\mathcal P_n}$ is well defined whenever $\|[h'(a), u']\|<\delta_n'$ for all $a\in \mathcal G_n'$ for any homomorphism $h'$ and unitary $u'$, and moreover, if $h\approx_{\delta_n'} h'$ on $\mathcal G_n'$, then
$$\mathrm{Bott}(h, u)|_{\mathcal P_n}=\mathrm{Bott}(h', u)|_{\mathcal P_n}.$$

Let $C_n'$ (in place of $C'$) with $K_1(C_n')=\Z^{r(n)}\oplus\mathrm{Tor}(K_1(C_n'))$, $\iota_n: C_n'\to C_n$, $\mathcal Q_n''\subseteq\Z^{r(n)}$ (in place of $\mathcal Q$), and $\eta_n$ (in place of $\delta$) be required by Theorem \ref{B1B2-alg} for $\mathcal G'_n$ (in place of $\mathcal F$), $\mathcal P'_n$ (in place of $\mathcal P$) and $\delta_n'/4$ (in place of $\ep$). One also fixes a finite set of generators of $K_1(C'_n)$ for each $n$. Without loss of generality, one may assume that $\mathcal Q_n''$ is the set of standard generators of $\Z^{r(n)}$.

Put $\delta_n=\min\{\eta_n, \delta_n'/2\}$.

By Lemma \ref{asy-01-tr0}, there are unitaries $v_n\in U(A)$ such that
$$\mathrm{ad}(v_n)\circ\phi_1\approx_{\delta_{n+1}/4} \phi_2\quad\textrm{on $\psi_{n+1, \infty}(\mathcal G'_{n+1})$},$$
$$\rho_A(\mathrm{bott}_1(\phi_2\circ\iota_n, v_n^*v_{n+1}))(x)=0\quad\textrm{for all $x\in\psi_{n+1, \infty}(K_1(C'_{n+1}))$},$$
and
$$\|[\phi_2(a), v^*_nv_{n+1}]\|<\delta_{n+1}/2\quad\textrm{for all $a\in\psi_{n+1, \infty}(\mathcal G'_{n+1})$}.$$
Then we have that
$$\mathrm{Bott}(\phi_1\circ\iota_{n+1}, v_{n+1}v^*_n)=\mathrm{Bott}(v_n^*(\phi_1\circ\iota_{n+1}) v_n, v_n^*(v_{n+1}v_n^*)v_n)=\mathrm{Bott}(\phi_2\circ\iota_{n+1}, v_n^*v_{n+1}).$$
In particular, for any $x\in(\psi_{n+1, \infty}\circ\iota_{n+1})_{*1}(K_1(C'_{n+1}))$, one has
$$\mathrm{bott}_1(v_n^*\phi_1v_n, v_n^*v_{n+1})(x)=\mathrm{bott}_1(\phi_2, v^*_nv_{n+1})(x).$$

By applying 10.4 and 10.5 of \cite{Lin-Asy}, without loss of generality, we may assume that
$\phi_1\circ\psi_{n+1, \infty}\circ\iota_{n+1}$ and $v_{n}$ define an element
$\gamma_n\in\mathrm{Hom}_{\Lambda}(\underline{K}(C'_{n+1}), \underline{K}(M_{\phi_1\circ\iota_{n+1}, \phi_2\circ\iota_{n+1}}))$ and $[\pi_0]\circ\gamma_n=[\iota_{n+1}]$. Moreover, $\gamma_n$ factors through $H_{n}:=[\psi_{n+1, \infty}\circ\iota_{n+1}](\underline{K}(C'_{n+1}))$. Thus, one may also regard $\gamma_n$ being defined on $H_n$.

Furthermore, by 10.4 and 10.5 of \cite{Lin-Asy}, without loss of generality, we may assume that
\begin{equation}\label{small-tr-001}
\tau(\log((\phi_2\circ\psi_{n+1, \infty}\circ\iota_{n+1}(z_j^*))\overline{v}^*_n (\phi_1\circ\psi_{n+1, \infty}\circ\iota_{n+1}(z_j)) \overline{v}_n))<\delta_{n+1}
\end{equation}
where $\{z_1, ..., z_{r(n)}\}\subseteq U(M_k(C'_{n+1}))$ induces a set of standard  generators of $\Z^{r(n)}\subseteq K_1(C'_{n+1})$ and $\overline{v}_n=\mathrm{diag}(\underbrace{v_n, ..., v_n}_k)$.

Since $\bigcup_{n=1}^\infty [\psi_{n+1, \infty}\circ\iota_{n+1}](\underline{K}(C_n'))=\underline{K}(C)$ and $[\pi_0]\circ\gamma_{n}=[\iota_{n+1}]$, one concludes
\begin{equation}\label{full-image}
\underline{K}(M_{\phi_1, \phi_2})=\underline{K}(SA)+\bigcup_{n=1}^\infty\gamma_n(H_n).
\end{equation}

By passing to a subsequence, one may assume that
$$\gamma_{n}(H_n)\subseteq\underline{K}(SA)+\gamma_{n+1}(H_{n+1}),\quad n=1, 2, ...$$


By 10.6 of \cite{Lin-Asy}, $\Gamma(\mathrm{Bott}(\phi_1, v_nv_{n+1}^*))|_{H_n}=(\gamma_{n}-\gamma_{n+1}\circ[\psi_{n+1}])|_{H_n}$ defines a homomorphism $\xi_n: H_n\to \underline{K}(SA)$. Then define a map $j_n:  \underline{K}(SA)\oplus H_n \to \underline{K}(SA)\oplus H_{n+1} $ by $$(x, y)\mapsto (x+\xi_n(y), [\psi_{n+1}](y)).$$ By \eqref{full-image}, the limit is $\underline{K}(M_{\phi_1, \phi_2})$.  One has the following diagram
$$
\begin{array}{cccccc}
0 \to  & \underline{K}(SA)  &\to \underline{K}(SA) & \bigoplus H_n  &\to H_n & \to 0
 \\
 & \downarrow_{=}  & \downarrow_{=} &\swarrow_{\xi_n} \hspace{0.2in} \downarrow_{[\psi_{n+1}]} & \downarrow_{[\psi_{n+1}]}\\
0 \to  & \underline{K}(SA)  &\to \underline{K}(SA) & \bigoplus H_{n+1}  &\to H_{n+1}  &\to 0.
\end{array}
$$

By the assumption that $\overline{R}_{\phi_1, \phi_2}=0$, the map $\theta$ also induces the following $$\ker R_{\phi_1, \phi_2}=\ker\rho_A\oplus K_1(C).$$

Define $\zeta_n=\gamma_{n+1}|_{H_n}$, $\theta_n=\theta\circ[\psi_{n+1, \infty}]$, and $\kappa_n=\zeta_n-\theta_n$. Note that
$$\theta_n=\theta_{n+1}\circ[\psi_{n+2}]$$ and
$$\zeta_n-\zeta_{n+1}\circ[\psi_{n+2}]=\xi_n.$$

Since
$[\pi_0]\circ(\zeta_n-\theta_n)=0$, $\kappa_n$ maps $H_n$ into $\underline{K}(SA)$. It follows that
\begin{eqnarray}\label{decomp}
\kappa_n-\kappa_{n+1} &=& \zeta_n-\theta_n-\zeta_{n+1}\circ[\psi_{n+2}]+\theta_{n+1}\circ[\psi_{n+2}]\\
&=&\zeta_n-\zeta_{n+1}\circ[\psi_{n+2}]=\xi_n.
\end{eqnarray}

It follows from 10.3 of \cite{Lin-Asy} that there are integers $N_1\geq 1$, a $\delta_{n+1}$-$\psi_{n+1}(\mathcal G'_{n+1})$-multiplicative map $$L_n: \psi_{n+1, \infty}\circ\iota_{n+1}(C'_{n+1})\to M_{1+N_1}(M_{\psi_1, \psi_2}),$$ a unital homomorphism $h_0: \psi_{n+1, \infty}\circ\iota_{n+1}(C'_{n+1})\to M_{N_1}(C)$, and a continuous path of unitaries $\{V_n(t): t\in[0, 3/4]\}$ of $M_{1+N_1}(A)$ such that $[L_n]|_{\mathcal P'_{n+1}}$ is well defined, $V_n(0)=1_{M_{1+N_1}(A)}$,
$$[L_n\circ\psi]|_{\mathcal P_n'}=(\theta\circ[\psi_{n+1, \infty}]+[h_0\circ\psi_{n+1, \infty}])|_{\mathcal P'_n},$$
$$\pi_t\circ L_n\circ\psi_{n+1, \infty}\approx_{\delta_{n+1}/4} \mathrm{ad} V_n(t)\circ((\phi_1\circ\psi_{n+1, \infty})\oplus (h_0\circ\psi_{n+1, \infty}))$$
on $\mathcal G_{n+1}$ and $t\in[0, 3/4]$, and
$$\pi_t\circ L_n\circ\psi_{n+1, \infty}\approx_{\delta_{n+1}/4} \mathrm{ad} V_n(3/4)\circ((\phi_1\circ\psi_{n+1, \infty})\oplus (h_0\circ\psi_{n+1, \infty}))$$
on $\mathcal G_{n+1}$ and $t\in(3/4, 1)$, and
$$\pi_1\circ L_n\circ\psi_{n+1, \infty}\approx_{\delta_{n+1}/4} (\phi_1\circ\psi_{n+1, \infty})\oplus (h_0\circ\psi_{n+1, \infty})$$
on $\mathcal G_{n+1}$.
Note that
${R}_{\phi_1, \phi_2}(\theta(x))=0$ for all $x\in(\psi_{n+1, \infty})_{*1}(K_1(C_{n+1}))$. As computed in 10.4 of \cite{Lin-Asy},
\begin{equation}\label{0-tr-001}
\tau(\log((\phi_2(z)\oplus h_0(z))^* V_n^*(3/4)(\phi_1(z)\oplus h_0(z))V_n(3/4)))=0
\end{equation}
for $z=(\psi_{n+1, \infty}\circ\iota_{n+1})_{*1}(y)$, where $y$ in the fixed set of generators of $K_1(C'_{n+1})$ and for all $\tau\in T(A)$.

Define $W'_n=\mathrm{diag}(v_n, 1)\in M_{1+N_1}(A)$. Then
$$\mathrm{Bott}((\phi_1\oplus h_0)\circ\psi_{n+1, \infty}\circ\iota_{n+1}, W_n'(V_n(3/4))^*)$$
defines a homomorphism $\tilde{\kappa}_n\in\mathrm{Hom}_{\Lambda}(\underline{K}(C'_{n+1}), \underline{K}(SA)).$

By \eqref{small-tr-001}, one has
$$\tau(\log( (\phi_2\oplus h_0)\circ\psi_{n+1, \infty}\circ\iota_{n+1}(z_j)^*\tilde{V}_n^*(\phi_1\oplus h_0)\circ\psi_{n+1, \infty}\circ\iota_{n+1}(z_j) \tilde{V}_n ))<\delta_{n+1}$$
for $j=1, 2, ..., r(n)$, where $\tilde{V}_n=\mathrm{diag}(\overline{v}_n, 1)$.
Then, by \eqref{0-tr-001}, one has
$$\rho_A(\tilde{\kappa}_n(z_j))(\tau)<\delta_{n+1},\quad j=1, 2, ..., r(n).$$
It then follows from Theorem \ref{B1B2-alg} that there is a unitary $w'_n\in U(A)$ such that
$$\|[\phi_1(a), w'_n]\|<\delta'_{n+1}/4,\quad\forall a\in\psi_{n+1, \infty}(\mathcal G_{n+1}),$$ and
$$\mathrm{Bott}(\phi_1\circ\psi_{n+1, \infty}\circ\iota_{n+1}, w_n')|_{\underline{K}(C'_{n+1})}=-\tilde{\kappa}_n.$$ Put $w_n=v_n^*w_n'v_n$. One has
$$\mathrm{Bott}(\phi_2\circ\psi_{n+1, \infty}\circ\iota_{n+1}, w_n)|_{\underline{K}(C_{n+1}')}=-\tilde{\kappa}_n|_{\underline{K}(C_{n+1}')}.$$

It follows from 10.6 of \cite{Lin-Asy} that
$$\Gamma(\mathrm{Bott}(\phi_1\circ\psi_{n+1, \infty}, w'_n))=-\kappa_n\quad\textrm{and}\quad \Gamma(\mathrm{Bott}(\phi_1\circ\psi_{n+2, \infty}, w'_{n+1}))=-\kappa_{n+1},$$
where $\Gamma$ is defined in 10.6 of \cite{Lin-Asy}.
One also has
$$\Gamma(\mathrm{Bott}(\phi_1\circ\psi_{n+1, \infty}, v_nv_{n+1}^*))|_{H_n}=\zeta_n-\zeta_{n+1}\circ[\psi_{n+2}]=\xi_n.$$
Then, by \eqref{decomp}, one has
$$-\mathrm{Bott}(\phi_1\circ\psi_{n+1, \infty}, w_n')+\mathrm{Bott}(\phi_1\circ\psi_{n+1, \infty}, v_nv^*_{n+1})+\mathrm{Bott}(\phi_1\circ\psi_{n+1, \infty}, w'_{n+1})=0.$$

Define $u'_n=v_nw_n^*$, $n=1, 2, ...$ Then, $$\mathrm{ad}(u'_n)\circ\phi_1\approx_{\delta_{n+1}'/2} \phi_2, \quad\forall a\in\psi_{n+1, \infty}(\mathcal G_{n+1}),$$
and
\begin{eqnarray}\label{Bott-trivial}
&&\mathrm{Bott}(\phi_2\circ\psi_{n+1, \infty}, (u'_n)^*u'_{n+1}) \\
& =& \mathrm{Bott}(\phi_2\circ\psi_{n+1, \infty}, w_nv_n^*v_{n+1}w_{n+1}^*) \nonumber\\
&=&\mathrm{Bott}(\phi_2\circ\psi_{n+1, \infty}, w_n)+\mathrm{Bott}(\phi_2\circ\psi_{n+1, \infty}, v_n^*v_{n+1})+\mathrm{Bott}(\phi_2\circ\psi_{n+1, \infty}, w^*_{n+1}) \nonumber\\
&=&\mathrm{Bott}(\phi_1\circ\psi_{n+1, \infty}, w_n')-\mathrm{Bott}(\phi_1\circ\psi_{n+1, \infty}, v_nv^*_{n+1})-\mathrm{Bott}(\phi_1\circ\psi_{n+1, \infty}, w'_{n+1})\nonumber\\
&=&0.\nonumber
\end{eqnarray}


In what follows, we shall construct unitaries $\{s_n\}\subseteq A$ such that
\begin{equation}\label{ak-01}
\|[\phi_2\circ\psi_{n+1, \infty}(a), s_n]\|<\delta_{n+1}'/2,\quad\forall a\in \mathcal G'_{n+1}
\end{equation}
\begin{equation}\label{ak-02}
\mathrm{Bott}(\phi_2\circ\psi_{n+1, \infty}, s_n)|_{\mathcal P_{n+1}'}=0,
\end{equation} and
\begin{equation}\label{ak-03}
\mathrm{dist}(g^{\phi_2\circ\psi_{n+1, \infty}}_{x_{n+1, j}, s^*_ns_{n+1}}, g^{\phi_2\circ\psi_{n+1, \infty}}_{x_{n+1, j}, (u'_n)^*u'_{n+1}})<\gamma_{n+1}'/2.
\end{equation}
(Recall that $x_{n, j}=[p_{n, j}]-[q_{n, j}]$, $j=1, ..., m(n)$ and $n=1, ...$, and $\{x_{n, 1}, ..., x_{n, m(n)}\}$ is free.)

Put $s_1=1$, and assume that $s_1, ..., s_n$ has been constructed. Let us construct $s_{n+1}$. Define the map $\Xi'_n: \Z^{m(n+1)}\to  U_{l(n+1)}(A)/CU_{l(n+1)}(A)$ by
$$\Xi'_n(x_{n+1, j})= g^{\phi_2\circ\psi_{n+1, \infty}}_{x_{n+1, j}, s_n (u'_n)^*u'_{n+1}},\quad j=1, ..., m(n+1)$$
with the map $\phi_2\circ\psi_{n+1, \infty}$ in the place of $\phi$ in \eqref{bu-notation}.

Note that, by \eqref{Bott-trivial}, $\mathrm{Bott}(\phi_2\circ\psi_{n+1, \infty}, (u'_n)^*u'_{n+1})=0$. Then, together with \eqref{ak-02}, one has that $\mathrm{Bott}(\phi_2\circ\psi_{n+1, \infty}, s_n(u'_n)^*u'_{n+1})=0$. In particular, this implies that
$${[g^{\phi_2\circ\psi_{n+1, \infty}}_{x_{n+1, j}, s_n (u'_n)^*u'_{n+1}}]=0 \,\,\, {\rm in}\,\,\, K_1(A), \quad j=1, ..., m(n+1).}$$
Therefore, the image of map $\Xi_n'$ is in $U_0(M_{l(n+1)}(A))/CU(M_{l(n+1)}(A))$. That is, $$\Xi'_n: \Z^{m(n+1)}\to U_0(M_{l(n+1)}(A))/CU(M_{l(n+1)}(A)).$$

Recall that there are fixed decomposition $[\psi_{n+1, n+2}](\Z^{m(n+1)})=\Z^{k(n+1)}\oplus\mathrm{Tor}([\psi_{n+1, n+2}](\Z^{m(n+1)}))$ (for some integer $k(n+1)$) and a fixed lifting of $\Z^{k(n+1)}$ in $\Z^{m(n+1)}$ for each $n$. Also recall that $\{y_{n+1, 1}, y_{n+1, 2}, ..., y_{n+1, k(n+1)}\}$ is a fixed set of generators for $\Z^{k(n+1)}$, and $\{y_{n+1,1}', y_{n+1, 2}',..., y_{n+1, k(n+1)}'\}$ are their liftings in $\Z^{m(n+1)}$. Then define the map $\Xi_n:\Z^{k(n+1)}\to U_0(M_{l(n+1)}(A))/CU(M_{l(n+1)}(A))$ by $$\Xi_n(y_j)=\Xi'_n(y'_j),\quad j=1, ..., k(n+1).$$

Let $\ep''_n>0$ be arbitrary (which will be fixed later).  Applying Theorem \ref{BB-exi} to $C_{n+2}$ (in place of $C$),  $[\psi_{n+1, n+2}](\mathbb Z^{m(n+1)})$ (in place of $G$),$A$, $\mathcal G'_{n+2}$ (in place of $\F$), $\mathcal P'_{n+2}$ (in place of $\mathcal P$), $\ep''_n$ (in place of $\ep$ and in place of $\gamma$), and $\Xi_n$ (in place of $\Gamma$), there is a unitary $s_{n+1}\in A$ such that
\begin{equation}\label{ak-001}
\|[\phi_2\circ\psi_{n+1, \infty}(a), s_{n+1}]\|<\ep''_n,\quad\forall a\in \mathcal G'_{n+1},
\end{equation}
\begin{equation}\label{ak-002}
\mathrm{Bott}(\phi_2\circ\psi_{n+2, \infty}, s_{n+1})|_{\mathcal P_{n+2}'}=0,
\end{equation}
and
\begin{equation}\label{ak-003}
\mathrm{dist}(g^{\phi_2\circ\psi_{n+2, \infty}}_{y_{n+1, j}, s_{n+1}}, \Xi_n(y_{n+1, j}))<\ep''_n, \quad j=1, ..., k(n+1).
\end{equation}

 Note that it follows from \eqref{ak-002} that each element $g^{\phi_2\circ\psi_{n+1, \infty}}_{x_{n+1, j}, s_{n+1}}$ is in $U_0(M_{l(n+1)}(A))/CU(M_{l(n+1)}(A))$, and hence each element $g_{y_{n+1, j}', s_{n+1}}^{\phi_2\circ\psi_{n+1, \infty}}$ is also in $U_0(M_{l(n+1)}(A))/CU(M_{l(n+1)}(A))$.

By choosing $\ep''_n<\delta'_{n+1}/2$ sufficiently small,  it follows from \eqref{ak-001} that the unitary $s_{n+1}$ satisfies \eqref{ak-01}. Since $\pi([\psi_{n+1, n+2}(x_{n+1, j})])$ is in the subgroup generated by $\{y_{n+1, 1}, ..., y_{n+1, k(n+1)}\}$, where $\pi$ is the projection map from $[\psi_{n+1, n+2}](\Z^{m(n+1)})$ to $\Z^{k(n+1)}$, by choosing $\ep''_n< \frac{\gamma_{n+1}'}{8\sum_{i, j}|c_{i, j}^{n+1}|}$, it follows from \eqref{ak-003} and \eqref{ak-tor} that for any $j=1, ..., m(n+1)$,
\begin{eqnarray*}
&& \mathrm{dist}(g^{\phi_2\circ\psi_{n+1, \infty}}_{x_{n+1, j}, s_{n+1}}, \Xi'_n(x_{n+1, j}))\\
 &\leq & \mathrm{dist}(g^{\phi_2\circ\psi_{n+1, \infty}}_{x_{n+1, j}, s_{n+1}},   \prod_{j=1}^{k(n+1)} (g_{y_{n+1, j}', s_{n+1}}^{\phi_2\circ\psi_{n+1, \infty}})^{c_{i, j}^{(n+1)}}) + \\
 && \mathrm{dist}(\prod_{j=1}^{k(n+1)} (g_{y_{n+1, j}', s_{n+1}}^{\phi_2\circ\psi_{n+1, \infty}})^{c_{i, j}^{(n+1)}}, \Xi'_n(x_{n+1, j}))\\
 &< &\frac{\gamma_{n+1}'}{8}+ \mathrm{dist}(\prod_{j=1}^{k(n+1)} (g_{y_{n+1, j}, s_{n+1}}^{\phi_2\circ\psi_{n+2, \infty}})^{c_{i, j}^{(n+1)}}, \Xi_n(\sum_{j=1}^{k(n+1)} c_{i, j}^{(n+1)} y_{n+1, j}))\\
 &<&\frac{\gamma_{n+1}'}{8}+\sum_{i, j}|c_{i, j}^{n+1}|\epsilon_n''<\frac{\gamma_{n+1}'}{4}.
\end{eqnarray*}

In other words,
$$\mathrm{dist}(g^{\phi_2\circ\psi_{n+1, \infty}}_{x_{n+1, j}, s_{n+1}}, g^{\phi_2\circ\psi_{n+1, \infty}}_{x_{n+1, j}, s_n (u'_n)^*u'_{n+1}})<\gamma'_{n+1}/4, \quad j=1, ..., m(n+1).$$ Hence
$$\mathrm{dist}(g^{\phi_2\circ\psi_{n+1, \infty}}_{x_{n+1, j}, s_n^*s_{n+1}}, g^{\phi_2\circ\psi_{n+1, \infty}}_{x_{n+1, j}, (u'_n)^*u'_{n+1}})<\gamma'_{n+1}/2, \quad j=1, ..., m(n+1),$$
which verifies \eqref{ak-03}.
Therefore, one obtains the sequence of unitaries $(s_n)$ satisfying \eqref{ak-01}, \eqref{ak-02} and \eqref{ak-03}.

Define $u_n=u_n's_n^*$, $n=1, 2, ...$ Then it follows from \eqref{ak-01} and \eqref{ak-02} that
\begin{equation}\label{hp-cond-1}
\|[\phi_2\circ\psi_{n+1, \infty}, u^*_nu_{n+1}]\|<\delta_n',
\end{equation}
and
\begin{equation}\label{hp-cond-2}
\mathrm{Bott}(\phi_2\circ\psi_{n+1, \infty}, u^*_nu_{n+1})|_{\mathcal P_{n+1}'}=0.
\end{equation}

It also follows from \eqref{ak-03} that
$$\mathrm{dist}(g^{\phi_2\circ\psi_{n+1, \infty}}_{x_{n+1, j}, s_n(u'_n)^*u'_{n+1}s_{n+1}^*}, \overline{1_A})<\gamma_{n+1}', \quad j=1, ..., m(n+1),$$
which is
\begin{equation}\label{hp-cond-3}
\mathrm{dist}(g^{\phi_2\circ\psi_{n+1, \infty}}_{x_{n+1, j}, u_n^*u_{n+1}}, \overline{1_A})<\gamma_{n+1}', \quad j=1, ..., m(n+1).
\end{equation}
Moreover, it also follows from the definition of $\Delta_n$ such that
\begin{equation}\label{hp-cond-4}
\mu_{\tau\circ\phi_2\circ\psi_{n, \infty}}(O_a)\geq\Delta_n(a),\quad \forall \tau\in\mathrm{T(A)},
\end{equation} where $O_a$ is any open ball in $X_n$ with radius $a\geq \eta'_n$.

With \eqref{hp-cond-1}, \eqref{hp-cond-2}, \eqref{hp-cond-3} and \eqref{hp-cond-4}, one applies Theorem \ref{hmtp-H} to obtain a path of unitaries $\{z_n(t): t\in [0, 1]\}$ in $A$ such that $$z_n(0)=0,\quad z_n(1)=u_n^*u_{n+1},$$ and $$\|[z(t), \phi_2\circ\psi_{n+1, \infty}]\|<1/2^{n+1},\quad\forall t\in[0, 1].$$

Define $$u(t+n-1)=u_nz_{n+1}(t),\quad t\in(0, 1],$$ and then $\{z(t); t\in[0, \infty)\}$ is a continuous path of unitary in $A$.

Note that
\begin{eqnarray*}
\mathrm{ad} u(t+n-1)\circ\phi_1&\approx_{\delta'_{n+1}}& \mathrm{ad}(z_{n+1}(t))\circ\phi_2\\
&\approx_{1/2^{n+1}}& \phi_2
\end{eqnarray*}
on $\F_{n+1}$ for all $t\in(0, 1)$.
It then follows that
$$\lim_{t\to\infty} u^*(t)\phi_1(a) u(t)=\phi_2(a)$$
for all $a\in C$, as desired.
\end{proof}

Let $C$ and $A$ be two unital \CA s. Recall that (see10.2 of \cite{Lin-Asy})
$$
H_1(K_0(C), K_1(A)):=\{x\in K_1(A): h([1_C])=x \,\,\,
{\rm for\,\,\, some}\,\,\, h\in \mathrm{Hom}(K_0(C),K_1(A))\}.
$$

\begin{lem}\label{StrongL}
Let $C$ be a unital AH-algebra and let $A$ be a unital separable simple
\CA\, with $\tr(A)\le 1.$ Suppose that $\phi,\psi: C\to A$ are two unital monomorphisms. Suppose that $\{{\cal F}_n\}$ is an increasing sequence
of finite subsets of $C$ such that $\cup_{n=1}^{\infty} {\cal F}_n$ is dense in $C,$ and suppose that $\{{\cal P}_n\}$ is an increasing sequence of finite subsets of $K_1(C)$ such that its union is $K_1(C).$ Suppose also that there is a sequence of decreasing positive numbers $\dt_n>0$ with $\sum_{n=1}^{\infty}\dt_n<\infty$ and a sequence of unitaries $\{u_n\}\subset A$ such that
\beq\label{StrL-1}
{\rm Ad}\, u_n\circ \phi\approx_{\dt_n}\psi\,\,\,{\rm on}\,\,\, {\cal F}_n\andeqn\\
\rho_A({\rm bott}_1(\psi, u_n^*u_{n+1})(x)=0\tforal x\in {\cal P}_n.
\eneq
Then
we may further require that $u_n\in U_0(A)$ if
$H_1(K_0(C), K_1(A))=K_1(A).$

\end{lem}

\begin{proof}
The proof  is exactly the same as that of Lemma 10.4 of \cite{Lnclasn}.
Note that, we will apply the second part of \ref{B1B2-alg} instead of 6.3 of \cite{Lnclasn}.
\end{proof}

\begin{thm}\label{StrongT}
Let $C$ be a unital AH-algebra and let $A$ be a unital separable simple
\CA\, with $\tr(A)\le 1.$ Suppose that $H_1(K_0(C), K_1(A))=K_1(A)$ and suppose that $\phi, \psi: C\to A$ are two unital monomorphisms which are asymptotically unitarily equivalent. Then they are strongly asymptotically unitarily equivalent, i.e., there exists a continuous path of unitaries
$\{u(t): t\in [0,\infty)\}\subset U(A)$ such that
$$
u(0)=1_A\andeqn \lim_{t\to\infty} u(t)^*\phi(c)u(t)=\psi(c)\tforal c\in C.
$$
\end{thm}

\begin{proof}
The proof is exactly the same as that of Theorem 10.5 in \cite{Lnclasn}.
However, we apply \ref{StrongL} instead of Lemma 10.4 of \cite{Lnclasn} as needed in the proof of \cite{Lnclasn}.
\end{proof}

\begin{cor}\label{StrongC}
Let $X$ be a compact metric space and let $A$ be a unital separable simple \CA\, with $\tr(A)\le 1.$ Suppose that $\phi, \psi: C\to A$ are two unital
monomorphisms. Then $\phi$ and $\psi$ are strongly asymptotically unitarily equivalent if and only if
\beq\label{StrC-1}
[\phi]=[\psi]\,\,\,{\rm in}\,\,\, KK(C(X), A), \,\phi^{\ddag}=\psi^{\ddag},\\
\tau\circ\phi=\tau\circ \psi\tand {\overline{R_{\phi, \psi}}}=0.
\eneq

\end{cor}

\begin{proof}
Note that $K_0(C(X))=(\Z\cdot [1_{C(X)}])\oplus G$ for some abelian subgroup $G$ of $K_0(C(X)).$ For each $x\in K_1(A),$ define
a \hm\, $h: K_0(C(X))\to K_1(A)$ by
$h([1_{C(X)}])=x$ and $h|_G=0.$ In other words, one has that
$H_1(K_0(C), K_1(A))=K_1(A).$
\end{proof}

\begin{prop}\label{MP1}
Let $C$ be a unital amenable \CA\, satisfying the UCT and let $A$ be a unital separable simple \CA\, with $\tr(A)\le 1.$  Suppose that
$\phi$ and $\psi$ are two unital monomorphisms. Suppose also that
\beq\label{MP1-1}
[\phi]=[\psi]\,\,\,{\rm in}\,\,\,KL(C,A),\\
\tau\circ \phi=\tau\circ \psi\tforal \tau\in T(A)\andeqn\\
R_{\phi, \psi}(K_1(M_{\phi, \psi}))\subset \rho_A(K_0(A)).
\eneq
Then
\beq\label{MP1-2}
\phi^{\dag}=\psi^{\dag}.
\eneq
\end{prop}

\begin{proof}

Let $u\in M_l(C)$ be a unitary, where $l\ge 1$ is an integer.
Let $z\in M_l(M_{\phi,\psi})$ be a unitary
which is piecewise smooth on $[0,1]$ such that
$\pi_0\circ z=\phi(u)$ and $\pi_1\circ z =\psi(u).$
Let $G$ be a finitely generated subgroup of $K_1(C)$ which contains
$[u].$
By the assumption, there is an injective \hm\, $\theta_G: G\to K_1(M_{\phi,\psi})$ such that
\beq\label{MP1-3}
(\pi_0)_{*1}\circ \theta_G={\rm id}_{G}\andeqn
R_{\phi, \psi}\circ \theta_G \in  \textrm{Hom}(G, \rho_A(K_0(A))).
\eneq
It follows that there exists projections
$p, q\in M_{l'}(A)$ such that
\beq\label{MP1-4}
\theta([u])=[zv]\in K_1(M_{\phi, \psi}),
\eneq
where $v(t)=(e^{2\pi i t}p+(1-p))(e^{-2\pi i t}p+(1-p))\in M_{l'}(M_{\phi, \psi}).$
To simplify the notation, without loss of generality, we may assume
that $l=l'.$
By (\ref{MP1-3}),
\beq\label{MP1-5}
R_{\phi, \psi}([zv])\in \rho_A(K_0(A)).
\eneq
Since $R_{\phi, \psi}([v])\in \rho_A(K_0(A)),$ one computes that
\beq\label{MP1-6}
R_{\phi, \psi}([z])\in \rho_A(K_0(A)).
\eneq
Now let $w(t)\in C([0,1], A)$ be a unitary which is piecewise smooth such that $w(0)=\psi(u)^*\phi(u)$ and $w(1)=1_{M_l(A)}.$
Then
\beq\label{MP1-7}
\psi(u)w(t)\in M_l(M_{\phi, \psi}).
\eneq
Moreover $[z]=[\psi(u)w]$ in $K_1(M_{\phi,\psi}).$
It follows that, for any $\tau\in T(A),$
\beq\label{MP1-8}
{1\over{2\pi i}}\int_0^1 \tau({d(w(t))\over{dt}} w^*(t)))dt &=&
{1\over{2\pi i}}\int_0^1 \tau(\psi(u) {d(w(t))\over{dt}} w^*(t))\psi(u)^*)dt\\
&=&{1\over{2\pi i}}\int_0^1 \tau({d(\psi(u)w(t))\over{dt}})(\psi(u)w(t)^*)dt\\
&=& R_{\phi, \psi}([z])(\tau).
\eneq
Thus, by (\ref{MP1-6}), there exists $x\in K_0(A)$ such that
\beq\label{MP1-8+}
\mathrm{Det}([w])(\tau) =\rho_A(x)(\tau)
\eneq
for all $\tau\in T(A).$
It follows from a result of P. Ng (\cite{Ng}) that
$$
\psi(u^*)\phi(u)\in DU(M_l(A)).
$$
Since this holds for all unitaries $u\in M_l(C),$ it follows
that
$$
\phi^{\dag}=\psi^{\dag}.
$$
\end{proof}

\begin{cor}\label{MC1}
Let $C$ be a unital AH-algebra and let A be a unital separable simple C*-algebra with $TR(A)\leq1$. Suppose that $\phi_1, \phi_2: C\to A$ are two unital monomorphisms. Then
$\phi$ and $\psi$ are asymptotically unitarily equivalent if and only if
\beq\label{MC1-1}
[\phi_1]=[\phi_2]\ \textrm{in}\ KK(C, A),\\
\tau\circ \phi=\tau\circ \psi\tforal \tau\in T(A)\tand\\
\overline{R}_{\phi_1, \phi_2}=0.
\eneq

\end{cor}

\begin{proof}
We only need to show the ``if part" of the statement.
It follows from \ref{MP1} that, in addition, one has
\beq\label{MC1-2}
\phi^{\dag}=\psi^{\dag}.
\eneq
This of course implies that
$
\phi^{\ddag}=\psi^{\ddag}.
$
Then \ref{MT1} applies.
\end{proof}

\begin{thm}\label{MT2}
Let $C$ be a unital  AH-algebra and let $A$ be a unital simple \CA\, with $TR(A)\le 1.$  Suppose that $\phi,\psi: C\to A$ are two unital
monomorphisms such that
\beq\label{MT2-1}
[\phi]=[\psi]\,\,\,{\rm in}\,\,\, KK(C,A),\\\label{MT2-1+1}
\tau\circ \phi=\tau\circ \psi, \tforal \tau\in T(A), \andeqn\\\label{MT2-1+2}
\phi^{\dag}=\psi^{\dag},
\eneq
then
$\phi$ and $\psi$ are asymptotically unitarily equivalent, provided that one of the following holds:
\begin{enumerate}
\item $K_1(C)$ is finitely generated, or
\item $K_0(A)$ is finitely generated, or
\item the short exact sequence
$$
0\to {\rm kre}\rho_A\to K_0(A)\to \rho_A(K_0(A))\to 0
$$
 splits.
 \end{enumerate}
\end{thm}

\begin{proof}
Let $C$ and $A$ be as in the statement. Suppose that $\phi,\, \psi: C\to A$ are two unital monomorphisms which satisfy the assumptions (\ref{MT2-1}),
(\ref{MT2-1+1}) and (\ref{MT2-1+2}).
In particular, (\ref{MT2-1+2}) implies that
\beq\label{MT2-2}
\phi^{\ddag}=\psi^{\ddag}.
\eneq
Since $[\phi]=[\psi],$ there exists a splitting map
$\theta: \underline{K}(C)\to \underline{K}(M_{\phi,\psi})$ such that
\beq\label{MT2-3}
\theta\circ [\pi_0]=[{\rm id}_C].
\eneq
Let $u\in M_l(C)$ be a unitary for some integer $l\ge 1.$
Let $z\in M_l(M_{\phi, \psi})$ be a unitary
such that $z(0)=\phi(u)$ and $z(1)=\psi(u).$ Moreover, we may assume
that $z$ is piecewise smooth.
Define $z_1(t)=\psi(u)^*z(t)$ for $t\in [0,1].$ Then $z_1$ is a piecewise smooth and continuous path of unitaries in $A$ such that
$z_1(0)=\psi(u)^*\phi(u)$ and $z_1(1)=1_{M_l}.$
It follows from (\ref{MT2-1+2}) that
\beq\label{MT2-4}
{1\over{2\pi i}}\int_0^1 \tau({dz_1(t)\over{dt}}z_1(t)^*)dt\in \rho_A(K_0(A)),
\eneq
where $\tau\in T(A).$  One then easily computes that
\beq\label{MT2-5}
R_{\phi, \psi}([z])\in \rho_A(K_0(A)).
\eneq
On the other hand,
there is a projection
$p\in M_{l'}(A)$ such that the following holds:
\beq\label{MT2-6}
\theta([u])=[zv],\,
\eneq
where $v(t)=e^{2\pi it}p+(1_{M_{l'}}-p)$ for all $t\in [0,1].$
To simplify the notation, without loss of generality, we may assume that
$l'=l.$  It follows that
\beq\label{MT2-7}
R_{\phi, \psi}([zv])=R_{\phi, \psi}([z])+R_{\phi, \psi}([v])
\in \rho_A(K_0(A)).
\eneq
It follows that
\beq\label{MT2-8}
R_{\phi, \psi}\circ \theta \in \mathrm{Hom}(K_1(C), \rho_A(K_0(A))).
\eneq
In all three cases (1), (2) and (3),
there exists a \hm\, $\lambda_0: R_{\phi, \psi}\circ \theta(K_1(C))\to K_0(A)$ such that
\beq\label{MT2-9}
\rho_A\circ \lambda_0={\rm id}_{R_{\phi, \psi}\circ \theta(K_1(C))}.
\eneq
Define $\lambda=\lambda_0\circ R_{\phi, \psi}\circ \theta.$
So $\lambda$ is a \hm\, from $K_1(C)$ into $K_0(A).$
Define, by viewing $K_0(A)$ as a subgroup of $K_1(M_{\phi, \psi}),$
\beq\label{MT2-10}
\theta_1=\theta-\lambda.
\eneq
Then
\beq\label{MT2-11}
R_{\phi, \psi}\circ \theta_1=0.
\eneq
It follows that
\beq\label{MT2-12}
\overline{R_{\phi, \psi}}=0.
\eneq
The theorem then follows from \ref{MT1}.
\end{proof}

\begin{cor}\label{C2}
Let $X$ be a finite CW-complex and let $A$ be a unital simple \CA\, with finite tracial rank. Suppose that
$\phi, \psi: C(X)\to A$ are two unital monomorphisms. Then
$\phi$ and $\psi$ are asymptotically unitarily equivalent if and only if
\beq\label{C2-1}
[\phi]=[\psi]\,\,\,\,{\rm in}\,\,\, KK(C,A),\\
\tau\circ \phi=\tau\circ \psi\tforal \tau\in T(A)\andeqn\\
\phi^{\dag}=\psi^{\dag}.
\eneq
\end{cor}

\begin{rem}\label{MR}
{\rm
We would point out that the assumptions
in \ref{MT1} is more sensitive than those in (\ref{MT2-1}), (\ref{MT2-1+1}) and (\ref{MT2-1+2}), in general.

Let $A$ be a unital simple AF-algebra with $K_0(A)$ given by
a non-splitting short exact sequence
\beq\label{MR-0}
0\to G\to K_0(A)\to \Q\to 0,
\eneq
where $G$ is a countable abelian group and where the order of an element is determined by its image in $\Q.$
In particular, $A$ has a unique tracial state $\tau$ and
$\rho_A(K_0(A))=\Q.$
Let $C$ be a unital simple \CA\, of tracial rank zero with
$K_1(C)=\Q\oplus {\rm Tor}(K_1(C))$ which also satisfies the UCT.
Let $\kappa\in KK(C,A)^{++}$ such that $\kappa([1_C])=[1_A].$
Then there exists a unital monomorphism $\phi: C\to A$ such that
$[\phi]=\kappa.$ Let $\lambda=\phi_T: T(A)\to T(C)$ be the affine continuous map induced by
$\phi.$ Let $\gamma: K_1(C)\to \rho_A(K_0(A))$ be an isomorphism as abelian group.
It follows from 4.8 of \cite{L-N} that there exists a unital monomorphism
$\psi: C\to A$ such that $[\psi]=\kappa=[\phi],$
$\psi_T=\lambda=\phi_T$ and
there exists a splitting map $\theta: \underline{K}(C)\to \underline{K}(M_{\phi, \psi})$ such that
\beq\label{MR-1}
R_{\phi, \psi}\circ \theta=\gamma +\gamma_0,
\eneq
where $\gamma_0\in {\cal R}_0.$ We may write
$\gamma_0=\rho_A\circ f,$ where $f: K_1(C)\to K_0(A)$  is a \hm.
It follows from \ref{MP1} that
\beq\label{MR-2}
\phi^{\dag}=\psi^{\dag}.
\eneq
However, there is no \hm\, $\lambda_1: K_1(C)\to K_0(A)$ such
that
$$
R_{\phi, \psi}\circ \theta=\rho_A\circ \lambda_1.
$$
Otherwise, $\eta=(\lambda_1-f)\circ \gamma^{-1}$ would split
the short exact sequence (\ref{MR-0}),
since
\beq
\rho_A\circ \eta &=&\rho_A\circ (\lambda_1-f)\circ \gamma_1^{-1}\\
&=& (R_{\phi, \psi}\circ \theta-\rho_A\circ f)\circ \gamma_1^{-1}\\
&=& (\gamma+\gamma_0-\gamma_0)\circ \gamma^{-1}={\rm id}_{\rho_A(K_0(A))}.
\eneq
In other words,
$$
\overline{R_{\phi, \psi}}\not=0.
$$
}

\end{rem}


\bibliographystyle{plain}

\end{document}